\documentclass{my_aims}

\usepackage{breqn}
\usepackage{amsmath,amssymb,amsthm}
\usepackage{amsmath,mleftright}
\usepackage{xparse}
\usepackage{hyperref}
\usepackage[noabbrev,capitalise,nameinlink]{cleveref}
\hypersetup{colorlinks={true},linkcolor={blue},citecolor=blue,urlcolor=cyan}
\NewDocumentCommand{\evalat}{sO{\big}mm}{%
	\IfBooleanTF{#1}
	{\mleft. #3 \mright|_{#4}}
	{#3#2|_{#4}}%
}
\def\typeofarticle{Research article}

\def\ppages{5364--5391}
\def\DOI{10.3934/mbe.2021272}
\def\Received{25 April 2021}
\def\Accepted{15 June 2021}
\def\Published{17 June 2021}

\usepackage[pagewise]{lineno}
\usepackage{bbm}

\numberwithin{equation}{section}

\newcommand{\mref}[3][cyan]{\hypersetup{linkcolor=#1}\cref{#2}{#3}\hypersetup{linkcolor=blue}}
\newtheorem{theorem}{Theorem}
\newtheorem{proposition}{Proposition}
\newtheorem{remark}{Remark}
\newtheorem{lemma}{Lemma}

\begin{document}

\title{Farming awareness based optimum interventions for crop pest control}

\author{Teklebirhan Abraha\affil{1}, 
Fahad Al Basir\affil{2}, 
Legesse Lemecha Obsu\affil{1} 
and Delfim F. M. Torres\affil{3,}\corrauth} 

\shortauthors{T. Abraha, F. Al Basir, L. L. Obsu, D. F. M. Torres}

\address{\addr{\affilnum{1}}{Department of Mathematics,
Adama Science and Technology University, Adama, Ethiopia}
\addr{\affilnum{2}}{Department of Mathematics,
Asansol Girls' College, West Bengal 713304, India}
\addr{\affilnum{3}}{Center for Research and Development
in Mathematics and Applications (CIDMA),\\
Department of Mathematics,
University of Aveiro, 3810-193 Aveiro, Portugal}}

\corraddr{Email: delfim@ua.pt; Tel: +351-234-370-668; Fax: +351-234-370-066.}
	

\begin{abstract}
We develop a mathematical model, based on
a system of ordinary differential equations,
to the upshot of farming alertness in crop
pest administration, bearing in mind plant biomass,
pest, and level of control.
Main qualitative analysis of the proposed mathematical model,
akin to both pest-free and coexistence equilibrium points
and stability analysis, is investigated. We show that all solutions
of the model are positive and bounded with initial conditions
in a certain significant set. The local stability of pest-free
and coexistence equilibria is shown using the Routh--Hurwitz criterion.
Moreover, we prove that when a threshold value is less than one, then
the pest-free equilibrium is locally asymptotically stable.
To get optimum interventions for crop pests, that is,
to decrease the number of pests in the crop field,
we apply optimal control theory and find the corresponding
optimal controls. We establish existence of optimal controls
and characterize them using Pontryagin's minimum principle.
Finally, we make use of numerical simulations to illustrate the
theoretical analysis of the proposed model, with and without control
measures.
\end{abstract}

\keywords{mathematical modeling of ecological systems;
Holling type-II functional response;
stability;
Hopf-bifurcation;
optimal control;
Pontryagin's minimum principle;
numerical simulations.}

\maketitle


\section{Introduction}

Pest control is a worldwide problem in agricultural and forest ecosystem management,
where mathematical modeling has an important impact \cite{MyID:478,MR4207019,MR4240711}.
Broad-spectrum chemical pesticides have been used in abundance in the containment
and annihilation of pests of medical, veterinary, agricultural, and environmental importance.
Although chemical plant defense plays a significant function in modern agricultural practices,
it is at rest viewed as a profit-induced poisoning of the surroundings. The non-degradable
chemical residues, which construct to damaging stages, are the core cause of health
and environmental hazards and most of the present hostility toward them. The imperative
role of microbial pesticides in integrated pest management,
is well-known in agriculture, forestry, and public fitness
\cite{MR4114006,MR4152828,MR4129647}. As included pest administration,
biological pesticides give noticeable pest control
reliability in the case of plants \cite{IJB}. The employ of viruses alongside insect pests,
as pest control agents, is seen worldwide, in particular in
North America and European countries \cite{Naranjo}.

Agricultural practitioners should be awake of the threats accessible by pests and diseases
and of fitting steps that may be taken to prevent their incidence and to undertake
them should they become problematic. Responsiveness campaigns, in meticulous through
radio or TV, are required so that people will gain trustworthiness on a biological
control approach. Farmers, in their awareness, can stay the crop beneath surveillance
and, so, if properly trained, they will send out bio-pesticides or fit in
productive to build the pest vulnerable to their bio-agents \cite{Khan}.

Precise and pertinent information about plants and their pests is necessary
for people occupied in crop growing. The role of electronic media is crucial
for observing the farming society rationalized and by providing pertinent
agricultural information \cite{IJB}. Accessible pesticide information campaigns
help farmers to be aware of the grave threats that insect killers have on human
health and the environment, and to reduce negative effects \cite{risks}.
Accepting consciousness programs, planned to teach farmers', fallout
in improved comprehensive growth for the cultivars and also for the farmers.
Farmers be trained the use and hazards of pesticides mostly by oral announcement.
Self aware farmers occupy considerably improved agronomic perform,
safeguarding the health and reducing environmental hazards.
Therefore, alertness is vital in crop pest administration \cite{Yang2014}.

Damage of plants by pest plague is a solemn international fret
in not only farming fields, but also in afforests ecosystems. This problem,
linked with pests, has been realized since the crop growing of harvest started.
About 42\% of world's food afford is exhausted because of pests. In former days,
chemical pesticide was used to control pests, but intense use of chemical pesticides
in farming produces a lot of side effects. Pest renaissance and secondary pest
outbursts are two of the key problems. Mathematical models have great reimbursements
for relating the dynamics of crop pest management. Quite a few models have been formulated
and analyzed to explain the dynamics of plant disease diffusion and charging their
controls \cite{JTB}. For instance, authors in \cite{LeBellec}
consider how allowing surroundings for connections among farmers, researchers,
and other factors, can donate to reduce existing problems related with the crop.
Al Basir et al., describe the participation of farming communities
in Jatropha projects for biodiesel construction and guard of plants from mosaic disease,
using a mathematical model to predict the growth of renewable power resources \cite{mma}.

In \cite{EC2}, Chowdhury et al. derive a model for the control of pests
by employing biological pesticides. Furthermore, they also include
optimal control theory to reduce the cost in pest administration due to bio-pesticides.
For more on optimal control theory to eradicate the number of parasites in agro ecosystems,
see the pioneering work \cite{MR0513616,MR2899157,MR3660471}. In \cite{ridm},
Al Basir and Ray investigate the dynamics of vector-borne plant disease dynamics
influenced by farming awareness. In \cite{SPA}, Pathak and Maiti develop a mathematical model
on pest control with a virus as a control agent. They consider
the mutual relations going on an ecosystem where a virus influences
a pest population nosh on a plant, the final creature unaltered by the virus.
They prepare a time-delayed model and prove a blend of detailed analysis of the model.
In concrete, they show the positivity and boundedness of solutions,
the dynamical behaviors of the model in the lack and occurrence of time-delay.
In \cite{JTB}, Al Basir et al. prospect a mathematical model for recitation
of the level of wakefulness in pest control, together with biological pesticides,
which are exactly the most ordinary control mechanisms used in integrated pest
administration based control outlines. They assume the speed of flattering aware
is comparative to the number of vulnerable pests in the meadow. They prepare
the model as more realistic, bearing in mind the time delay due to assess of pest
in the field. Non-negativity of the solutions is analyzed, by finding the invariant region.
Stability analyses of the systems has been studied using qualitative methods.
The obtained analytical results are then verified through numerical simulations.
In \cite{IJB}, Al Basir develop a new model by employing delayed differential equations
to learn the vibrant of crop pests interacting system prejudiced via a responsiveness movement.
The author tacit the rate of becoming aware is relative to the number of susceptible pests
within the grassland \cite{IJB}.

Here, a mathematical model is formulated to defend crops
through awareness movements, modeled via saturated terms,
and an optimal control problem for bio-pesticides is posed and solved.
More precisely, a compartmental model is developed with ordinary differential equations
to study the force of farming awareness-based optimum interventions for crop pest control
(\mref{sec2}). It is assumed that with the pressure of bio-pesticide,
a susceptible pests population becomes infected.  
Infected pests can harass the plant but the rate 
is very smaller than susceptible pests. Hence,
we suppose that infected pests cannot guzzle the plant biomass. Non-negativity and
boundedness of the solutions are proved by computing the feasible region (\mref{sec3.1}).
Stability analysis of the dynamics is carried out (\mref{sec3.2}).
Optimal control techniques are then investigated (\mref{sec4}).
In concrete, optimal control theory is practiced in the dynamics
to reduce the price of pest administration. Our analytical results are
illustrated by computing numerical simulations
(\mref{sec5}). Lastly, we end up the paper
with discussion and conclusions (\mref{sec6}).


\section{Model formulation and description}
\label{sec2}

We consider four populations into our mathematical model:
the plants biomass $X(t)$, the susceptible pests $S(t)$,
the infected pests $I(t)$, and the awareness level $A(t)$.
The next assumptions are done to define the dynamical system:
\begin{itemize}
	
\item[--] The sway of the biological pesticides, susceptible pest
inhabitants turn into infected. The infected pests can assault
the crop but with a very lesser to the susceptible pest and, therefore,
we assume that infected pests cannot devour the crop biomass.
In this technique, the crop is kept from pest molest.
	
\item[--] Owing towards the limited size of harvest field, we take
for granted logistical increase for the thickness of crop biomass,
by means of net enlargement rate $r$ and carrying ability $K$.
	
\item[--] Crops obtain influenced via pests. In this way,
rooting significant plant biomass decreases.
Crop pest contact may be considered as a Holing type II functional response.
	
\item[--] Susceptible attacks the plant, resulting extensive plant lessening.
Once one infects the susceptible pest using pesticides, the attacked pests
will be controlled. Here it is assumed that the aware people will accept
bio-pesticides for the plant pest control, as it has fewer disadvantages
and is also surroundings welcoming. Bio-pesticides are used to contaminate
the susceptible pest. Infected pest are assumed to have an extra death rate
due to infection. We also additionally suppose that infected pests
cannot devour the plants.
	
\item[--] Let $\alpha$ be the utilization rate of pest.
We assume that there is a pest infection rate, $\lambda$,
because of human awareness connections and movement such as use
of bio-pesticides, modeled through the common crowd exploit term
$\lambda A S$. We denote by $d$ the natural death rate of pest and
by $\delta$ the extra death rate of infected pest owing
to awareness population motion.
	
\item[--] We suppose that the level of conscious of people will
enlarge on a rate almost similar to the figure of susceptible pest
per plant noticed in a farming field. There might be vanishing
of attention in this operation. We denote by $\eta$ the rate
of fading of interest of aware people.
\end{itemize}

Susceptible pest consume the crop, thereby 
causing considerable crop reduction. If we infect the 
susceptible pest by pesticides, then the pest attack 
can be reduced as infected pest is lesser harmful.

One example of bio-pesticides is a 
Nuclear Polyhidrosis Virus (NPV) from the Baculovirus family, 
which is often host-specific and usually fatal for pests. 
Once susceptible pests feed on plants that have been treated 
by this or similar viruses, they will ingest virus particles, 
which will make them infected. In this article, we consider 
a scenario, where plants are treated (using spraying or soaking) 
with such bio-pesticides that only affect insect pests causing 
in them a persistent and fatal infection.

Susceptible pests, if once infected, 
are assumed lesser harmful for crop biomass. The aware 
people infect the pest population and the infected population 
does not recover or become immune. The infection rate is governed 
by the so-called mass-action incidence.

The crop-pest interaction is studied with 
the Michaelis--Menten type (or Holling type II) functional 
response \cite{SPA}. In the type II functional response, 
the rate of prey consumption by a predator rises as prey 
density increases, but eventually levels off at a plateau 
(or asymptote) at which the rate of consumption remains constant 
regardless of increases in prey density. Here we realize 
huge number of pest for the crop filed (saturation occur). 
Type-II functional response is characterized by a decelerating 
intake rate, which follows from the assumption that the consumer 
is limited by its capacity to process food. Type II functional 
response is often modeled by a rectangular hyperbola, for instance 
as by Holling's disc equation, which assumes that processing 
of food and searching for food are mutually exclusive behaviors.

The level of awareness, $A(t)$, can be raised by 
seeing the crops or simply talking about its health and benefit, 
through direct interactions or by visual inspecting the crops. 
This occurs at rate $\sigma$ \cite{IJB}.

Using the above mentioned assumptions, we get the following
mathematical model:
\begin{equation}
\label{state1}
\begin{cases}
\displaystyle \frac{dX}{dt}
= r X\left[1 - \frac{X}{K}\right]
- \frac{\alpha XS}{c+X}- \frac{\phi \alpha XI}{c+X},\\[0.3cm]
\displaystyle \frac{dS}{dt}
=  \frac{m_1 \alpha XS}{c+X} -  \frac{\lambda AS}{a+A} -dS,\\[0.3cm]
\displaystyle \frac{dI}{dt}
= \frac{m_2 \phi\alpha XI}{c+X} +  \frac{\lambda AS}{a+A}- (d+\delta)I,\\[0.3cm]
\displaystyle \frac{dA}{dt} =\gamma+ \sigma(S+I) - \eta A,
\end{cases}
\end{equation}
with given initial situations $X(0)>0$, $S(0)>0$, $I(0)>0$, $A(0)>0$.

Here $\alpha$ is the consumption rate of pests on crops,
the infected pest may also molest the plant but at a lesser rate, $\phi\alpha$,
with $\phi<1$, $a$ and $c$ are the half saturation constants, $m_1$, $m_2$
are the ``conversion efficiency'' of the susceptible and the infected pest,
respectively, i.e., how proficiently can pests utilize plant supply.
As pesticide affecting pests have lesser efficiency, we consider $m_1>m_2$.
By $\gamma$ we denote the increase of level from global advertisement by radio,
TV, etc. It is clear to assume that all the parameters are non-negative.


\section{Model analysis}
\label{sec3}

Now, some essential properties of the solutions of system \eqref{state1} are given.
In concrete, we show non-negativeness, invariance, and boundedness of solutions.
From a biological point of view, these mathematical properties of our system \eqref{state1}
are crucial to the well-posedness of the model.


\subsection{Positivity of solutions and the invariant region}
\label{sec3.1}

Feasibility and positivity of the solutions are the basic properties
of system \eqref{state1} to be shown in this section.
Our result explains the region in which the solution of the equations
is biologically relevant.

All state variables should remain non-negative, since they represent
plant and pest population. The feasible region is, therefore, given by		
\[
\Omega=\left\{(X,S,I,A)\in\mathbb{R}_+^4 :
X \geq 0, S \geq 0,I \geq 0,A \geq 0\right\}.
\]
To prove the positivity of the solutions 
of system \eqref{state1}, we use 
the following  lemma.

\begin{lemma}
\label{lemma1}
Any solution of a differential equation 
\begin{equation}
\label{eq:de:le1}
\frac{dX}{dt} = X\psi(X, Y )
\end{equation}
is always positive.
\end{lemma}

\begin{proof}
A differential equation of form \eqref{eq:de:le1}
can be written as 
$\displaystyle \frac{dX}{X} = X\psi(X, Y )dt$. 
Integrating, we obtain that
$\ln X=C_0+\int\psi(X,Y)dt$, i.e., one has
$X=C_1e^{\int\psi(X,Y)dt}>0$ with $C_1>0$. 
\end{proof}

\begin{proposition}
The solutions of system \eqref{state1}, 
together with their initial conditions,
remain non-negative for all $t>0$.
\end{proposition}

\begin{proof}
We use \mref{lemma1}~to prove the positivity 
of the first two equations. It follows from 
the first equation of system \eqref{state1} that 
\[
\frac{dX}{dt}=X\left[r\left(1-\frac{X}{K}\right)
-\frac{\alpha S+\phi\alpha I}{c+X}\right]
\Rightarrow 
\left[r\left(1-\frac{X}{K}\right)
-\frac{\alpha S+\phi\alpha I}{c+X}\right]dt.
\]
Hence,
\[
\ln X=D_0+\int_{0}^{T} \left[r\left(1-\frac{X}{K}\right)
-\frac{\alpha S+\phi\alpha I}{c+X}\right]dT
\]
so that
\[
X(T)=D_1e^{\displaystyle \int_{0}^{T}\left[r\left(1-\frac{X}{K}\right)
-\frac{\alpha S+\phi\alpha I}{c+X}\right]dT}>0
\quad (\because D_1=e^{D_0}).
\]
From the second equation of system \eqref{state1}, 
we have
\[
\frac{dS}{dt}=S\left[\frac{m_1\alpha X}{c+X}-\frac{\lambda A}{a+A}-d\right]
\Rightarrow \frac{dS}{S}=\left[\frac{m_1\alpha X}{c+X}
-\frac{\lambda A}{a++A}-d\right]dt.
\]
Therefore,
\[
\ln S=K_0+\int_{0}^{T} 
\left(\frac{m_1\alpha X}{c+X}
-\frac{\lambda A}{a+A}-d\right)dT
\]
and
\[
S(T)=K_1e^{\displaystyle \int_{0}^{T} \left(\frac{m_1\alpha X}{c+X}
-\frac{\lambda A}{a+A}-d\right)dT}>0.
\]
To show that $I$ and $A$ are non-negative, 
consider the following sub-system of \eqref{state1}:
\begin{equation}
\label{subsystem}
\begin{cases}
\displaystyle \frac{dI}{dt}=\frac{m_2\phi\alpha XI}{c+X}
+\frac{\lambda AS}{a+A}-(d+\delta)I,\\[0.3cm]
\displaystyle \frac{dA}{dt}=\gamma+\sigma(S+I)-\eta A.
\end{cases}
\end{equation}
To show the positivity of \(I(t)\),
we do the proof by contradiction.
Suppose there exists \(t_0\in(0,T)\) 
such that $I(t_0)=0$, $I'(t_0)\leq0$ 
and \(I(t)>0\) for \(t\in[0,t_0)\). Then, 
\(A_0>0\) for \(t\in[0,t_0)\). If this is not to be the case, 
then there exists \(t_1\in[0,t_0)\) such that 
$A(t_1)=0$, $A'(t_1)\leq 0$ and \(A(t)>0\) for \(t\in[0,t_0)\). 
Integrating the third equation of system \eqref{state1} gives
\begin{equation*}
\begin{aligned}
I(t)&=I(0)\,\exp\left(m_2\,\phi\,\alpha\,\int_{0}^{t}\left( \frac{X(\tau)}{c+X(\tau)}d\tau\right)-(d+\delta)t\right)\\
&+\left[\exp\left(m_2\,\phi\,\alpha\,\int_{0}^{t}\left( \frac{X(\tau)}{c+X(\tau)}d\tau\right)-(d+\delta)t\right)\right]\\
&\times\left[\int_{0}^{t}\frac{\lambda\,A(\tau)S(\tau)}{a+A(\tau)}d\tau\,
\exp\left((d+\delta)t-m_2\,\phi\,\alpha\,\int_{0}^{t}\left(
\frac{X(\tau)}{c+X(\tau)}d\tau\right)\right)\right]>0 
\quad\mbox{for}\quad t\in[0,t_1].
\end{aligned}
\end{equation*}
Then, \(A'(t_1)=\gamma+\sigma(S(t_1)+I(t_1))>0\). 
This is a contradiction. Hence, \(I(t)>0\) for all \(t\in[0,t_0)\).
Finally, from the second equation of subsystem \eqref{subsystem}, 
we have
\[
\frac{dA}{dt}=\gamma+\sigma(S+I)-\eta A.
\] 
Integration gives
\[
A(t)=A(0)\,e^{\eta t}+e^{\eta t}\,
\int_{0}^{t}\left(\gamma+\sigma(S(\tau)
+I(\tau))\right)e^{-\eta t}dt>0,
\]
that is, \(A(t)>0\) for all \(t\in(0,T)\).
\end{proof}

\mref{thm1}~gives a region $\mathcal{D}$ that attracts all solutions 
initiating inside the interior of the positive octant.

\begin{theorem}
\label{thm1}
Let $M :=\max\{X(0),K\}$ and
$$
\mathcal{D}:=\left\{(X,S,I,A)\in \mathbb{R}^{4}_+: 0\leq X
\leq M ,0\leq X+S+I\leq \frac{(r+4d) M}{4d},0\leq A
\leq \frac{4rd+\sigma(r+4d)M}{4\eta rd}\right\}.
$$
Every solution of system \eqref{state1}
that starts in $\mathcal{D}$ is uniformly bounded.
\end{theorem}

\begin{proof}
From the first equation of system \eqref{state1}, we have
\begin{equation}
\label{eq 2}
\frac{dX}{dt} = r X\left[1 - \frac{X}{K}\right]
- \frac{\alpha XS}{c+X}- \frac{\phi \alpha XI}{c+X}
\leq r X\left[1 - \frac{X}{K}\right],
\end{equation}
which implies that $\displaystyle \frac{dX}{dt}\leq rX\left(1-\frac{X}{K}\right)$.
This is a separable ordinary differential equation and,
taking $X(0)=X_0$, we get
\[
\lim\limits_{t\rightarrow \infty}X(t)\leq M.
\]
Letting $W(t)=X(t)+S(t)+I(t)$ at any time $t$,
we get from the first three equations of \eqref{state1} that
\begin{equation*}
\begin{split}
\frac{dW(t)}{dt}
&= r X\left[1 - \frac{X}{K}\right]
- \frac{\alpha XS}{c+X}- \frac{\phi \alpha XI}{c+X}
+\frac{m_1 \alpha XS}{c+X}
+ \frac{m_2 \phi\alpha XI}{c+X} -dS- (d+\delta)I\\
&=r X\left[1 - \frac{X}{K}\right] - \left(\frac{\alpha XS}{c+X}
- \frac{m_1 \alpha XS}{c+X}\right)-\left(\frac{\phi\alpha XI}{c+X}
- \frac{m_2 \phi\alpha XI}{c+X}\right)
+ \underline{dX}-dS-(d+\delta)I-\underline{dX}\\
&= r X\left[1-\frac{X}{K}\right] +dX-d(S+I+X)-\delta I\\
&\leq r X\left[1-\frac{X}{K}\right] +dX-d(S+I+X)
=r X\left[1-\frac{X}{K}\right] +dX-dW\\
&\leq \frac{r}{4}K+dM-dW\\
&=\left(\frac{r+4d}{4}\right) M-dW,
\end{split}
\end{equation*}
that is, $\displaystyle \frac{dW}{dt}+dW\leq\left(\frac{r+4d}{4}\right) M$. Hence,
\begin{equation*}
\lim_{t\rightarrow\infty}W(t)\leq\left(\frac{r+4d}{4d}\right)M.
\end{equation*}
Note that $\displaystyle r X\left[1-\frac{X}{K}\right]$
is a quadratic expression in $X$ and its maximum value is
$\displaystyle \frac{r K}{4}$.
From the last equation of system \eqref{state1}, we get
\begin{equation*}
\begin{split}
\frac{dA}{dt}
&=\gamma+ \sigma(S+I) - \eta A\\
&\leq \gamma+\sigma\left(\frac{r+4d}{4d}\right)M-\eta A\\
&\leq \frac{4\gamma d+\sigma(r+4d)}{4d}M-\eta A,
\end{split}
\end{equation*}
that is,
$$
\frac{dA}{dt}+\eta A\leq \frac{4\gamma d+\sigma(r+4d)}{4d}M.
$$
We conclude that
$$
\lim\limits_{t\rightarrow\infty}A(t)
\leq \frac{dA}{dt}+\eta A\leq \frac{4\gamma d+\sigma(r+4d)}{4\eta d}M
$$
and, therefore, all solutions of system \eqref{state1}
are attracted to $\mathcal{D}$.
\end{proof}


\subsection{Equilibria and stability}
\label{sec3.2}

To get the fixed points, we put the right-hand side
of system \eqref{state1} equal to zero:
\begin{equation}
\label{eq3}
\begin{cases}
\displaystyle r X\left[1 - \frac{X}{K}\right] 
- \frac{\alpha XS}{c+X}- \frac{\phi \alpha XI}{c+X}=0,\\[0.3cm]
\displaystyle \frac{m_1 \alpha XS}{c+X} 
- \frac{\lambda AS}{a+A} -dS=0,\\[0.3cm]
\displaystyle \frac{m_2 \phi\alpha XI}{c+X} 
+ \frac{\lambda AS}{a+A}- (d+\delta)I=0,\\[0.3cm]
\gamma+ \sigma(S+I) - \eta A=0.
\end{cases}
\end{equation}
It follows from \eqref{eq3} that system \eqref{state1} has four equilibrium points:
\begin{itemize}
\item[(i)] The axial equilibrium point $E_0 = \left(0,0,0,\frac{\gamma}{\eta}\right)$;
\item[(ii)] The pest free equilibrium point $E_1 = \left(K,0,0,\frac{\gamma}{\eta}\right)$;
\item[(iii)] The susceptible pest free equilibrium point
$E_2 = \left(\bar{X},0,\bar{I},\bar{A}\right)$
with
\begin{equation*}
\begin{split}
\bar{X}&=\frac{c(d+\delta)}{ m_2\phi\alpha-(d+\delta)},\\
\bar{I}&=\frac{r(c+\bar{X})(K-\bar{X})}{\phi\alpha K}
=\frac{K\alpha\phi m_2-(K+c)(d+\delta)}{K\left(\alpha\phi m_2-d-\delta\right)^2}crm_2,\\
\bar{A}&=\frac{\gamma+\sigma \bar{I}}{\eta}=\frac{K\gamma\left(\phi\alpha m_2
-(d+\delta)\right)^2+K\phi\alpha\sigma m_2^2 cr
-(K+c)(d+\delta)crm_2\sigma}{\left(\phi\alpha m_2-(d+\delta)\right)^2},
\end{split}
\end{equation*}
that exists if, and only if,
\begin{equation}
\label{eq:iii:cond:HPFE}
K>\bar{X}\Rightarrow d+\delta<\frac{m_2\phi\alpha K}{c+K};
\end{equation}
\item[(iv)] The coexistence or endemic equilibrium point
$E^* = \left(X^*,S^*,I^*,A^*\right)\neq0$.
\end{itemize}
The condition \eqref{eq:iii:cond:HPFE} indicates that if the death rate of the infected pest is low,
then $E_2$ exists. Equivalently, it means that if the conversion factor $m_2$ of the infected pest
is high, then $E_2$ exists. Regarding the coexistence equilibrium,
$E^*$ is the steady state solution where pest persist in the crop biomass population.
It is obtained by setting each equation of the system \eqref{state1} equal to zero, that is,
\[
\frac{dX}{dt}=\frac{dS}{dt}=\frac{dI}{dt}=\frac{dA}{dt}=0.
\]
From the second equation of system \eqref{eq3}, we get
$$
\left(\frac{m_1 \alpha X}{c+X} -  \frac{\lambda A}{a+A} -d\right)S=0,
$$
which implies that
$$
\frac{m_1 \alpha X}{c+X} -  \frac{\lambda A}{a+A} -d=0.
$$
We obtain
\[
X^*=\frac{c(\lambda A+d(a+A))}{(m_1\alpha-d)(a+A)-\lambda A}.
\]
From the first equation of system \eqref{eq3}, we have
\[
r\left(1-\frac{X}{K}\right)-\frac{\alpha S}{c+X}-\frac{\alpha \alpha I}{c+X}=0
\]
and it follows that
\[
\frac{\alpha S}{c+X}+\frac{\phi\alpha I}{c+X}=r\left(1-\frac{X}{K}\right)=\frac{r(K-X)}{K}
\]
and
\[
\frac{\alpha S+\phi\alpha I}{c+X}=\frac{r(K-X)}{K}.
\]
Therefore,
\begin{equation}
\label{eq:star}
\alpha S+\phi\alpha I=\frac{r(K-X)(c+X)}{K}
\end{equation}
and from the last equation  of system \eqref{eq3} we have
\begin{equation}
\label{eq:star2}
\gamma+\sigma(S+I)-\eta A=0.
\end{equation}
Solving the system of equations \eqref{eq:star} and \eqref{eq:star2}
we get
\[
S^*=\frac{\alpha K\phi(\eta A-\gamma)-\sigma r(K-X^*)(c+X^*)}{\sigma\alpha K(\phi-1)}
\]
and
\[
I^*=\frac{r\sigma(K-X^*)(c+X^*)-\alpha K(\eta A^*-\gamma)}{\sigma\alpha K(\phi-1)}.
\]
We conclude that $E^*=(X^*,S^*,I^*,A^*)$ is the coexistence steady state where
\begin{eqnarray*}
X^*&=&\frac{c(\lambda A+d(a+A^*))}{(m_1\alpha-d)(a+A^*)-\lambda A^*},\\
S^*&=&\frac{\alpha K\phi(\eta A^*-\gamma)-\sigma r(K-X^*)(c+X^*)}{\sigma\alpha K(\phi-1)},\\
I^*&=&\frac{r\sigma(K-X^*)(c+X^*)-\alpha K(\eta A^*-\gamma)}{\sigma\alpha K(\phi-1)}
\end{eqnarray*}
and $A^*$ is the non-negative solution of equation
\[
f(A)=D_0A^4+D_1A^3+D_2A^2+D_3A+D_4=0
\]
whose coefficients are given by
\begin{equation*}
\begin{split}
D_0&=m_1+\frac{m_1(d-\delta)+\lambda m_1\delta-\phi m_2(d+\lambda)}{\alpha(\phi m_2-m_1)},\\
D_1&=\frac{\left[(K-2c)m_2\phi\alpha+(3c-K)\delta\right](d+\lambda)
+\alpha(2c-K)(d+\lambda+\delta)}{\alpha^2m_1(\phi m_2-m_1)},\\
D_2&=-m_1\lambda\gamma-\frac{\left[m_1\lambda-m_2\phi\lambda\right]\left[\alpha^2 m_1
K\gamma+\sigma r(\alpha Km_1-d)\right]}{\sigma m_1\alpha(\phi m_2-m_1)}
+\frac{[(d+\delta)m_1-m_2\phi d][\sigma r\lambda+\alpha^2 m_1
K\eta]}{\sigma m_1\alpha(\phi m_2-m_1)},\\
D_3&=\frac{[(d+\delta)m_1-m_2\phi d][\alpha^2m_1K\gamma
+\sigma r(\alpha Km_1-d)]}{\sigma m_1\alpha(\phi m_2-m_1)},\\
D_4&=\frac{\left[ac^2d^2(\phi-1)+ac\eta^2d\delta(\phi-1)
+c^2d\sigma r\delta(d+\lambda)\right]}{\sigma r m_1(\phi m_2-m_1)}.
\end{split}
\end{equation*}

The stability analysis for \eqref{state1} is studied
by linearization of the nonlinear system.
More precisely, we study the stability of an equilibrium point looking to the
eigenvalues of the corresponding Jacobian, which are functions of the model parameters.
The Jacobian matrix $J$ for system \eqref{state1} is given by
\[
J(X,S,I,A)=
\left[
\begin{array}{cccc}
\displaystyle \frac{\partial f_1}{\partial X}
& \displaystyle \frac{\partial f_1}{\partial S}
& \displaystyle \frac{\partial f_1}{\partial I} 
& \displaystyle \frac{\partial f_1}{\partial A}\\[0.3cm]
\displaystyle \frac{\partial f_2}{\partial X}
&\displaystyle \frac{\partial f_2}{\partial S}
&\displaystyle \frac{\partial f_2}{\partial I} 
&\displaystyle \frac{\partial f_2}{\partial A}\\[0.3cm]
\displaystyle \frac{\partial f_3}{\partial X}
&\displaystyle \frac{\partial f_3}{\partial S}
&\displaystyle \frac{\partial f_3}{\partial I} 
&\displaystyle \frac{\partial f_3}{\partial A} \\[0.3cm]
\displaystyle \frac{\partial f_4}{\partial X}
&\displaystyle \frac{\partial f_4}{\partial S}
&\displaystyle \frac{\partial f_4}{\partial I} 
&\displaystyle \frac{\partial f_4}{\partial A}
\end{array}
\right]
\]
\[
=\left[
\begin{array}{cccc}
\displaystyle r \left[1 - \frac{2X}{K}\right] - \frac{ \alpha cS}{(c+X)^2}
-\frac{\phi\alpha cI}{(c+X)^2}
& \displaystyle -\frac{\alpha X}{c+X}
& \displaystyle -\frac{\phi\alpha X}{c+X}& 0\\[0.3cm]
\displaystyle \frac{m_1\alpha cS}{(c+X)^2}&
\displaystyle \frac{m_1\alpha X}{c+X}
-\frac{\lambda A}{a+A}-d & 0 &
\displaystyle -\frac{\lambda a S}{(a+A)^2}\\[0.3cm]
\displaystyle \frac{m_2\phi\alpha cI}{(c+X)^2}&
\displaystyle \frac{\lambda A}{a+A} &
\displaystyle \frac{m_2\phi\alpha X}{c+X}-d-\delta&
\displaystyle \frac{\lambda aS}{(a+A)^2} \\[0.3cm]
0&\sigma&\sigma&-\eta
\end{array}
\right],
\]
where
\begin{eqnarray*}
f_1 &=& r X\left[1 - \frac{X}{K}\right]
- \frac{\alpha XS}{c+X}- \frac{\phi \alpha XI}{c+X},\\
f_2 &=&  \frac{m_1 \alpha XS}{c+X} -  \frac{\lambda AS}{a+A} -dS,\\
f_3 &=& \frac{m_2 \phi\alpha XI}{c+X} + \frac{\lambda AS}{a+A}- (d+\delta)I,\\
f_4 &=&\gamma+ \sigma(S+I) - \eta A.
\end{eqnarray*}

Next, we investigate the stability of each one of the four equilibrium points.

\begin{theorem}
System \eqref{state1} is always unstable around the axial equilibrium $E_0$.
\end{theorem}

\begin{proof}
The Jacobian matrix at the axial equilibrium $E_0$ is
\begin{equation}
\label{eqn6}
J\left(0,0,0,\frac{\gamma}{\eta}\right)=\left[
\begin{array}{cccc}
r  & 0 & 0& 0\\
0& -\frac{\lambda \gamma}{a\eta+\gamma}-d & 0& 0\\
0& \frac{\lambda \gamma}{a\eta+\gamma} & -d-\delta&0 \\
0&\sigma&\sigma&-\eta
\end{array}
\right],
\end{equation}
whose characteristic equation  is given by
\begin{equation}
\label{eqn7}
\left|\rho-J(E_0)\right|=\left|
\begin{array}{cccc}
\rho-r  & 0 & 0& 0\\
0& \rho-(\frac{\lambda \gamma}{a\eta+\gamma}-d) & 0& 0\\
0& -\frac{\lambda \gamma}{a\eta+\gamma} & \rho-(d+\delta)&0 \\
0&-\sigma&-\sigma&\rho+\eta
\end{array}
\right|=0.
\end{equation}
From
\[
(\rho-r)\left(\rho+\frac{\lambda\gamma}{a \eta+\gamma}
+d\right)(\rho+d+\delta)(\rho+\eta)=0,
\]
we obtain that the corresponding eigenvalues are
\[
\rho_1=r>0, \quad
\rho_2=-\eta<0,\quad
\rho_3=-d-\delta<0,\quad
\rho_4=\frac{-\lambda\gamma}{a\eta+\gamma}-d<0.
\]
Because $\rho_1=r>0$, we conclude that the axial
equilibrium $E_0$ is always unstable.
\end{proof}

In contrast with the axial equilibrium $E_0$,
which is always unstable, the pest free steady
state $E_1$ can be stable or unstable.
Precisely, \mref{stab:thm:ii}~asserts 
that the pest free steady state $E_{1}=(K,0,0,\frac{\gamma}{\eta})$ 
is locally asymptotically stable if
$d> \displaystyle\frac{m_1\alpha K}{c+K}-\frac{\lambda\gamma}{a\eta+\gamma}$ 
and $d>\displaystyle\frac{m_2 \phi\alpha K}{c+K}-\delta$.
This means $E_{1}=(K,0,0,\frac{\gamma}{\eta})$ is stable when
\[
d>\max\left\{\displaystyle\frac{m_1\alpha K}{c+K}\frac{\lambda\gamma}{a\eta+\gamma},
~\displaystyle\frac{m_2 \phi\alpha K}{c+K}-\delta\right\},
\]
which tells us that if the natural death rate $d$ of the pest population 
is high, then the system will approach to the pest free population. 
This, biologically, implies that the environment will be free 
of pest for that particular situation.

\begin{theorem}
\label{stab:thm:ii}	
The pest free steady state $E_1$ is stable if
\begin{equation}
\label{stabCond:thm:ii}
\frac{m_1\alpha K}{c+K}<\frac{\lambda\gamma}{a\eta+\gamma}+d
\quad \text{and} \quad
\frac{m_2\phi\alpha K}{c+K}<d+\delta
\end{equation}
and unstable if
\begin{equation}
\label{UnstabCond:thm:ii}
\frac{m_1\alpha K}{c+K}>\frac{\lambda\gamma}{a\eta+\gamma}+d
\quad \text{or} \quad \frac{m_2\phi\alpha K}{c+K}>d+\delta.
\end{equation}
\end{theorem}

\begin{proof}
The Jacobian matrix $J(E_1)$, at the pest free equilibrium
point $E_1=\left(K,0,0,\frac{\gamma}{\eta}\right)$, is given by
\[
J\left(K,0,0,\frac{\gamma}{\eta}\right)
=\left[
\begin{array}{cccc}
-r   & -\frac{\alpha K}{c+K} & -\frac{\phi\alpha K}{c+K}& 0\\
0& \frac{m_1\alpha K}{c+K}-\frac{\lambda \gamma}{a\eta+\gamma}-d & 0&0\\
0& \frac{\lambda \gamma}{a\eta+\gamma} & \frac{m_2\phi\alpha K}{c+K}-d-\delta&0 \\
0&\sigma&\sigma&-\eta
\end{array}
\right].
\]
The characteristic equation in $\rho$ at $E_1$ is
\[
|\rho I-J(E_1)|=\left|
\begin{array}{cccc}
\rho+r  & \frac{\alpha K}{c+K} & \frac{\phi\alpha K}{c+K}& 0\\
0& \rho-\left(\frac{m_1\alpha K}{c+K}
-\frac{\lambda \gamma}{a\eta+\gamma}-d\right) & 0&0\\
0& -\frac{\lambda \gamma}{a\eta+\gamma} &\rho
-\left(\frac{m_2\phi\alpha K}{c+K}-d-\delta\right)&0 \\
0&-\sigma&-\sigma&\rho+\eta
\end{array}
\right|=0,
\]
which gives
\[
(\rho+r)\left(\rho-\frac{m_1\alpha K}{c+K}+\frac{\lambda\gamma}{a\eta+\gamma}
+d\right)\left(\rho-\frac{m_2\phi\alpha K}{c+K}+d+\delta\right)(\rho+\eta) = 0.
\]
Thus, the eigenvalues are $-r$,
$\displaystyle \frac{m_1\alpha K}{c+K}-\frac{\lambda\gamma}{a\eta+\gamma}-d$,
$\displaystyle \frac{m_2\phi\alpha}{c+K}-d-\delta$, and $-\eta$.
Clearly, $J(E_1)$ has two negative eigenvalues, namely, $-r$ and $-\eta$.
Therefore, $E_1$ is stable or unstable, respectively if \eqref{stabCond:thm:ii}
or \eqref{UnstabCond:thm:ii} holds.
\end{proof}

The stability condition \eqref{stabCond:thm:ii}
for the pest free steady state $E_1$,
given by \mref{stab:thm:ii}, implies that
\[
\alpha K<\frac{\lambda\gamma(c+K)+d(a\eta+\gamma)(c+K)}{m_1(a\eta+\gamma)}
\]
and
\[
\alpha K<\frac{(c+K)(d+\delta)}{m_2\phi}.
\]
So, $E_1$ is stable if
$$
\alpha K<\min\left\{\frac{(c+K)(\lambda\gamma+d(a\eta+\gamma))}{m_1(a\eta+\gamma)},
\frac{(c+K)(d+\delta)}{m_2\phi }\right\}=:R.
$$
If we define
\[
R_{0}:= \frac{\alpha\,K}{R},
\]
then we can rewrite \mref{stab:thm:ii}~by saying that the pest free equilibrium
$E_1$ is locally asymptotically stable if $R_0<1$ and unstable if $R_0>1$.
The condition $R_0<1$ for the stability of the pest free equilibrium point $E_1$
indicates that if the death rates of the pest population are high, then the system
may stabilize to the pest free steady state $E_1$. Further, it can be noticed that the
existence of $E_2$ destabilizes $E_1$.

\begin{theorem}
\label{stab:thm:iii}	
The susceptible pest free equilibrium $E_2$
is locally asymptotically stable if
\begin{equation}
\label{stabCond:thm:iii}
\frac{m_1(d+\delta)(a+\bar{A})}{(a+\bar{A})(\lambda+d)
-\lambda a}<m_2\phi<\frac{(K+c)(d+\delta)}{\alpha(K-c)}.
\end{equation}
\end{theorem}

\begin{proof}
At the susceptible pest pest free 
fixed point $E_2=(\bar{X},0,\bar{I},\bar{A})$,
the Jacobian matrix $J(E_2)$ is given by
\[
J(E_2)=\left[
\begin{array}{cccc}
F_{11}  & -\frac{\alpha \bar{X}}{c+\bar{X}}
& -\frac{\phi\alpha \bar{X}}{c+\bar{X}}& 0\\
0& F_{22} & 0&0\\
\frac{m_2\phi\alpha c\bar{I}}{(c+\bar{X})^2}
& \frac{\lambda \bar{A}}{a+\bar{A}} & F_{33}&0 \\
0&\sigma&\sigma&-\eta
\end{array}
\right],
\]
where
\[
F_{11}=r\left(1-\frac{2\bar{X}}{K}\right)
-\frac{\phi\alpha c\bar{I}}{(c+\bar{X})^2},
\]
\[
F_{22}=\frac{m_1\alpha \bar{X}}{c+\bar{X}}
-\frac{\lambda \bar{A}}{a+\bar{A}}-d,
\]
and
\[
F_{33}=\frac{m_2\phi\alpha \bar{X}}{c+\bar{X}}-d-\delta.
\]
The characteristic equation in $\rho$ is given by
\[
\left|\rho I-J(E_2)\right|
= \left|
\begin{array}{cccc}
\rho-F_{11}  
& \displaystyle \frac{\alpha \bar{X}}{c+\bar{X}}
& \displaystyle \frac{\phi\alpha \bar{X}}{c+\bar{X}}& 0\\[0.3cm]
0& \rho-F_{22} & 0&0\\[0.3cm]
\displaystyle -\frac{m_2\phi\alpha c\bar{I}}{(c+\bar{X})^2}
& \displaystyle -\frac{\lambda \bar{A}}{a+\bar{A}} 
& \rho-F_{33}&0 \\[0.3cm]
0&-\sigma&-\sigma&\rho+\eta
\end{array}
\right|=0,
\]
which gives
\[
(\rho+\eta)(\rho-F_{22})\left[\rho^2-(F_{11}+F_{33})\rho+F_{11}F_{33}
+\frac{m_2\phi\alpha c\bar{I}\phi\alpha \bar{X}}{(c+\bar{X})^3}\right]=0,
\]
that is,
\[
(\rho+\eta)(\rho-F_{22})\left[\rho^2
+(-F_{11})\rho+\frac{m_2\phi\alpha
c\bar{I}\phi\alpha \bar{X}}{(c+\bar{X})^3}\right]=0
\]
or
\[
\left(\rho^2+B\rho+C\right)(\rho+\eta)(\rho-F_{22})=0,
\]
where $B=-F_{11}$ and
$C=\displaystyle
\frac{m_2\phi\alpha c\bar{I}\phi\alpha\bar{X}}{(c+\bar{X})^3}>0$.
The corresponding eigenvalues are
\[
\rho_{1,2}=\frac{-B\pm\sqrt{B^2-4C}}{2},
\quad \rho_3=F_{22}
\quad \text{and} \quad
\rho_4=-\eta.
\]
From the Routh--Hurwitz criteria, with a second degree polynomial, i.e,
$\rho^2+a_1\rho+a_2=0$, the necessary and sufficient condition
for the local stability of the system is that all eigenvalues
must have a negative real part (the condition $a_1>0$ and $a_2>0$ must hold).
Therefore, in our case, if $B>0$, then $\rho_1$ and $\rho_2$ are negative
(since $C>0$). This implies that $E_2$ is locally asymptotically stable
if $F_{11}<0$ and $F_{22}<0$:
\begin{equation*}
r\left(1-\frac{2\bar{X}}{K}\right)<\frac{\phi\alpha c\bar{I}}{(c+\bar{X})^2}
\Leftrightarrow
m_2\phi<\frac{(K+c)(d+\delta)}{\alpha(K-c)}
\end{equation*}
and
\begin{equation*}
F_{22}<0
\Leftrightarrow
\frac{m_1(d+\delta)(a+\bar{A})}{(a+\bar{A})(\lambda+d)-\lambda a}<m_2\phi.
\end{equation*}
Hence, $E_2$ is locally asymptotically stable provided \eqref{stabCond:thm:iii} holds.
\end{proof}

\begin{remark}
If we substitute $\bar{A}$ with
$$
\frac{K\gamma\left(\phi\alpha m_2-(d+\delta)\right)^2+K\phi\alpha\sigma m_2^2 cr
-(K+c)(d+\delta)crm_2\sigma}{\left(\phi\alpha m_2-(d+\delta)\right)^2},
$$
then we can express condition \eqref{stabCond:thm:iii}	
only in terms of the parameter values of the model.
\end{remark}

\mref{stab:thm:iii}~means that when the conversion
rate $m_2$ of the pest governs a moderate value and the
pest infection rate $\lambda$ is high, then it is expected
that the system will stabilize at the steady state
when all the pest become infected.

\newpage

We now investigate the stability of the fourth equilibrium point.

\begin{theorem}
\label{thm:we:def:Ci}
The coexistence steady state $E^*$ is locally asymptotically stable if
\begin{equation}
\label{eqxx}
C_2>0, \quad
C_3>0,\quad
C_4>0,\quad
C_1 C_2-C_3>0
\quad \text{ and } \quad
(C_1C_2-C_3)C_3-C_1^2C_4>0
\end{equation}
with
\begin{equation*}
\begin{split}
C_1&=-(F_{11}+F_{22}+F_{33})+\eta,\\
C_2&=F_{11}F_{22}+F_{11}F_{33}+F_{22}F_{33}-F_{11}\eta
-F_{22}\eta-F_{33}\eta+\frac{m_1\alpha cS^*\alpha X^*}{(c+X^*)^3}
+\frac{m_2\phi\alpha cI^*\phi\alpha X^*}{(c+X^*)^3},\\
C_3&=-F_{11}F_{22}F_{33}+F_{11}F_{22}\eta+F_{11}F_{33}\eta+F_{22}F_{33}\eta\\
&\quad - F_{33}\frac{m_1\alpha cS^*\alpha X^*}{(c+X^*)^2}
-F_{22}\frac{m_2\phi\alpha cI^*\phi\alpha X}{(c+X)^3}
+F_{22}\frac{\sigma\lambda aS^*}{(a+A)^2}-F_{33}\frac{\sigma\lambda aS^*}{(a+A)^2}\\
&\quad -\frac{m_1\alpha cS^*\lambda A^*\phi\alpha X^*}{(c+X^*)^3(a+A^*)^2}
-\frac{\lambda A^*\sigma\lambda aS^*}{(a+A^*)^3}+\frac{m_1\alpha cS^*\alpha
X^*\eta}{(c+X^*)^3}+\frac{m_2\phi\alpha c I^*\phi\alpha X^*\eta}{(c+X^*)^3},\\
C_4&=-F_{11}F_{22}F_{33}\eta-F_{11}F_{22}\sigma
\frac{\lambda aS}{(a+A)^2}+F_{11}F_{33}\sigma\frac{\lambda aS^*}{(a+A^*)^2}
+F_{11}\frac{\lambda A\sigma\lambda aS^*}{(a+A^*)^3}\eta\\
&\quad -F_{33}\frac{m_1\alpha cS^*\alpha X^*\eta}{(c+X^*)^3}
-F_{22}\frac{m_2\phi\alpha cI^*\phi\alpha X^*\eta}{(c+X^*)^3}
-\frac{m_1\alpha cS^*\sigma\alpha X^*\lambda aS^*}{(a+A^*)^2(c+X^*)^3}
-\frac{m_2\phi\alpha cI^*\sigma\alpha X^*\lambda aS^*}{(a+A^*)^2(c+X^*)^3}\\
&\quad + \frac{m_1\alpha cS^*\sigma\phi\alpha X^*\lambda aS^*}{(c+X^*)^3(a+A^*)^2}
+\frac{m_2\phi\alpha cS^*I^*\sigma\phi\alpha X^*\lambda a S^*}{(c+X^*)^3(a+A^*)^2}
-\frac{m_1\alpha c S^*\lambda A\phi\alpha X^*\eta}{(c+X^*)^3(a+A^*)}.
\end{split}
\end{equation*}
\end{theorem}

\begin{proof}
The Jacobian matrix $J(E^*)$, at the coexistence
equilibrium point $E^*$, is computed as
\[
J(E^*)=J(X^*,S^*,I^*,A^*)=\left[
\begin{array}{cccc}
F_{11}  
&\displaystyle -\frac{\alpha X^*}{c+X^*} 
& \displaystyle -\frac{\phi\alpha X^*}{c+X^*}& 0\\[0.3cm]
\displaystyle \frac{m_1\alpha cS^*}{(c+X^*)^2}
& F_{22} & 0
&\displaystyle \frac{\lambda aS^*}{(a+A^*)^2}\\[0.3cm]
\displaystyle \frac{m_2\phi\alpha cI^*}{(c+X^*)^2}
& \displaystyle \frac{\lambda A^*}{a+A^*} & F_{33}
&\displaystyle -\frac{\lambda aS^*}{(a+A^*)^2} \\[0.3cm]
0&\sigma&\sigma&-\eta
\end{array}
\right],
\]
where
\[
F_{11}=r\left(1-\frac{2X^*}{K}\right)
-\frac{\alpha cS^*}{(c+X^*)^2}
-\frac{\phi\alpha cI^*}{(c+X^*)^2},
\]
\[
F_{22}=\frac{m_1\alpha X^*}{c+X^*}-\frac{\lambda A^*}{a+A^*}-d,
\]
\[
F_{33}=\frac{m_2\phi\alpha X^*}{c+X^*}-d-\delta.
\]
The characteristic equation in $\rho$ for the
Jacobian matrix $J(E^*)$  is given by
\[
\left|\rho I-J(E^*)\right|
= \left|
\begin{array}{cccc}
\rho-F_{11}  
& \displaystyle \frac{\alpha X}{c+X^*}
& \frac{\phi\alpha X^*}{c+X^*}& 0\\[0.3cm]
\displaystyle -\frac{m_1\alpha cS^*}{(c+X^*)^2}&\rho- F_{22}
& 0&\displaystyle \frac{\lambda aS^*}{(a+A^*)^2}\\[0.3cm]
\displaystyle-\frac{m_2\phi\alpha cI^*}{(c+X^*)^2}
& \displaystyle -\frac{\lambda A^*}{a+A^*}
& \rho-F_{33}
&\displaystyle -\frac{\lambda aS^*}{(a+A^*)^2} \\[0.3cm]
0&-\sigma&-\sigma&-\eta
\end{array}
\right|=0,
\]
which gives
\begin{equation}
\label{char}
\rho^4+C_1\rho^3+C_2\rho^2+C_3\rho+C_4=0.
\end{equation}
Recognizing that $C_1>0$, applying the Routh--Hurwitz criterion,
and the conditions in \eqref{eqxx}, we conclude that the coexistence
equilibrium $E^*$ of system \eqref{state1} is locally
asymptotically stable if $C_2 > 0$, $C_3 > 0$, $C_4 > 0$,
$C_1C_2-C_3 > 0$, and $(C_1C_2-C_3)C_3-C_1^2C_4 > 0$;
and unstable otherwise.
\end{proof}

Now, let us discuss if the stability behavior of the system
at the coexistence steady state can be changed by varying
given parameters. We focus on the pest consumption
rate $\alpha$, which is considered as the most biologically significant parameter.
Hopf-bifurcation of the coexistence steady state $E^*$ may happen
if the auxiliary equation \eqref{char} has a couple of purely imaginary
eigenvalues for $\alpha=\alpha^*\in(0,\infty)$ with all the other eigenvalues
containing negative real parts. For the Hopf bifurcation to normally appear,
the transversality condition
\[
\left.\frac{dRe[\rho(\alpha)]}{d\alpha}\right|_{\alpha^*}\neq 0
\]
must be satisfied.

\begin{theorem}
Let $\Psi:(0, \infty)\rightarrow \mathbb{R}$ be
the continuously differentiable function of $\alpha$
defined by
\[
\Psi(\alpha):=C_1(\alpha)C_2(\alpha)C_3(\alpha)
-C_3^2(\alpha)-C_4(\alpha)C_1^2(\alpha),
\]
where the $C_i$, $i = 1, \ldots, 4$, are as in \mref{thm:we:def:Ci}.
The coexistence equilibrium $E^*$ of system \eqref{state1} enters
into a Hopf bifurcation at $\alpha=\alpha^*\in (0, \infty)$
if, and only if, the following conditions hold:
\begin{gather*}
C_2(\alpha^*) > 0, \quad C_3(\alpha^*) > 0,
\quad C_4(\alpha^*)>0, \quad C_1(\alpha^*)C_2(\alpha^*)-C_3(\alpha^*)>0,\\
\Psi(\alpha^*)=0, \quad
C_1^3C_2' C_3(C_1-3C_3)> 2(C_2C_1^2-2C_3^2)(C_3' C_1^2-C_1' C_3^2),
\end{gather*}
where we use primes $'$ to indicate derivatives with respect to parameter $\alpha$.
Moreover, at $\alpha=\alpha^*$, two characteristic eigenvalues $\rho(\alpha)$
are purely imaginary, and the remaining two have negative real parts.
\end{theorem}

\begin{proof}
By the condition $\Psi(\alpha^*)=0$, the characteristic
equation \eqref{char} can be written as
$$
\left(\rho^2+\frac{C_3}{C_1}\right)
\left(\rho^2+C_1\rho+\frac{C_1C_4}{C_3}\right)=0.
$$
If it has four roots, say $\rho_i$, $i=1,2,3,4$, with the pair of purely
imaginary roots at $\alpha=\alpha^*$ as $\rho_1=\bar{\rho_2}$, then we get
\begin{equation}
\label{sim}
\begin{split}
\rho_3+\rho_4&=-C_1,\\
\omega_0^2+\rho_3\rho_4&=C_2,\\
\omega_0^2(\rho_3+\rho_4)&=-C_3,\\
\omega_0^2\rho_3\rho_4&=C_4,
\end{split}
\end{equation}
where $\omega_0=Im \, \rho(\alpha^*)$. From the above, we have
$\omega_0=\sqrt{\frac{C_3}{C_1}}$. Now, if $\rho_3$ and $\rho_4$ are complex
conjugate, then, from \eqref{sim}, it follows that $2 Re \, \rho_3=-C_1$;
if they are real roots, then, by \eqref{char} and \eqref{sim}, $\rho_3<0$
and $\rho_4 < 0$. Now we verify the transversality condition.
As $\psi(\alpha^*)$ is a continuous function of all its roots,
then there exists an open interval $\alpha\in(\alpha^*-\epsilon,\alpha+\epsilon)$
where $\rho_1$ and $\rho_2$ are complex conjugate for $\alpha$. Suppose that
their general forms in this neighborhood are
\begin{eqnarray*}
\rho_1(\alpha)&=&\chi(\alpha)+i\upsilon(\alpha),\\
\rho_2(\alpha)&=&\chi(\alpha)-i\upsilon(\alpha).
\end{eqnarray*}
We want to verify the transversality conditions
\begin{equation}
\label{eq:TCs:pf}
\left.\frac{dRe[\rho_j(\alpha)]}{d\alpha}\right|_{\alpha=\alpha^*}\neq 0,
\quad j=1,2.
\end{equation}
Substituting $\rho_j(\alpha)=\chi(\alpha)\pm i\upsilon(\alpha)$ into
\eqref{char}, and calculating the derivative, we have
\begin{eqnarray*}
A(\alpha)\chi-B(\alpha)\upsilon'(\alpha)+C(\alpha)&=&0,\\
B(\alpha)\chi'+A(\alpha)\upsilon'(\alpha)+D(\alpha)&=&0,
\end{eqnarray*}
where
\begin{eqnarray*}
A(\alpha)&=&4\chi^3-12\mathcal{X}\upsilon^2+3C_1(\chi^2-\upsilon^2)
+2C_2\chi+C_3,\\
B(\alpha)&=&12\chi^2v+6C_1\chi\upsilon-4\mathcal{X}^3+2C\chi,\\
C(\alpha)&=&C_1\chi^3-3C_1'\chi\upsilon^2+C_2'(\chi^2-\upsilon^2)+C_3'\chi,\\
D(\alpha)&=&3C_1\chi^2\upsilon-C_1'\upsilon^3+C_2'\chi \upsilon+C_3'\chi.
\end{eqnarray*}
Hence, solving for $\chi'(\alpha^*)$, we get
\begin{eqnarray*}
\left.\frac{dRe[\rho(\alpha)]}{d\alpha}\right|_{\alpha=\alpha^*}
&=& \left.\chi'(\alpha)\right|_{\alpha=\alpha^*}\\
&=&-\frac{B(\alpha^*)D(\alpha^*)+A(\alpha^*)C(\alpha^*)}{A^2(\alpha^*)
+B^2(\alpha^*)}\\
&=&\frac{C_1^3C_2'C_3(C_1-3C_3)-2(C_2C_1^2-2C_3^2)(C_3'C_1^2-C_1'C_3^2)}{C_1^4
(C_1-3C_3)^2+4(C_2C_1^2-2C_3^2)^2}>0\\
&\Leftrightarrow&C_1^3C_2'C_3(C_1-3C_3)-2(C_2C_1^2-2C_3^2)(C_3'C_1^2
-C_1'C_3^2)>0\\
&\Leftrightarrow&C_1^3C_2'C_3(C_1-3C_3)>2(C_2C_1^2-2C_3^2)(C_3'C_1^2
-C_1'C_3^2).
\end{eqnarray*}
Thus, the transversality conditions \eqref{eq:TCs:pf} hold
and Hopf bifurcation occurs at $\alpha^*$.
\end{proof}


\section{Optimal control}
\label{sec4}

In this part, we extend the model system \eqref{state1} by incorporating
two time dependent controls $u_1(t)$ and $u_2(t)$. The first control
$u_1(t)$ represents the efficiency of pesticide that can be used, while
$u_2(t)$ characterizes the cost of the alertness movement. The objective
is to reduce the price of announcement for farming awareness via radio, TV,
telephony and other social media and the price regarding control measures.
Our goal is to find the optimal functions $u_1^*(t)$ and $u_2^*(t)$ by using
the Pontryagin minimum principle, as given in \cite{Fleming}. Therefore, our
system \eqref{state1} is modified to the induced state nonlinear dynamics
given by
\begin{equation}
\label{state2}
\begin{cases}
\displaystyle \frac{dX}{dt}
= r X\left[1 - \frac{X}{K}\right] - \frac{\alpha XS}{c+X}
- \frac{\phi \alpha XI}{c+X},\\[0.3cm]
\displaystyle \frac{dS}{dt} =
\frac{m_1 \alpha XS}{c+X} - u_1 \frac{\lambda AS}{a+A} -dS,\\[0.3cm]
\displaystyle \frac{dI}{dt}
= \frac{m_2 \phi\alpha XI}{c+X} + u_1 \frac{\lambda AS}{a+A}
- (d+\delta)I,\\[0.3cm]
\displaystyle \frac{dA}{dt} =u_2\gamma+ \sigma(S+I) - \eta A,
\end{cases}
\end{equation}
with $X(0)=X_0$, $S(0)=S_0$, $I(0)=I_0$ and $A(0)=A_0$ as initial conditions.
At this point, we need to reduce the number of pests and also the price of pest
administration by reducing the cost of pesticides and exploiting the stage
of awareness. Hence, we describe the price functional for the minimization
problem as
\begin{eqnarray}
\label{obj}
J(u_1(\cdot),u_2(\cdot))=\int^{t_f}_{0}
\left[A_1S^2(t)-A_2A^2(t)+\frac{1}{2}B_1u_1^2(t)
+\frac{1}{2}B_2u_2^2(t)\right] dt,
\end{eqnarray}
issued to the persuaded state control system \eqref{state2}, where the amounts
$A_1$ and $A_2$ represent the penalty multipliers on benefit of the cost,
and $B_1$ and $B_2$ stand for weighting constants on the profit of the price
of making. The terms $\frac{1}{2}B_1u_1^2(t)$ and $\frac{1}{2}B_2u_2^2(t)$
stand for the cost linked with pest managing and level of awareness.
We prefer a quadratic cost functional on the controls as an approximation
to the real nonlinear functional that depends on the supposition that the cost takes
a nonlinear form. As a consequence, we prevent bang-bang or singular optimal
control cases \cite{Fleming}. The  target here is to find the optimal controls
$(u_1^*,u_2^*)$ such that
\begin{eqnarray}
\label{u}
J(u_1^*,u_2^*)= \min(J(u_1,u_2) \, | \, (u_1,u_2)\in U),
\end{eqnarray}
where
\begin{equation}
\label{control:set}
\begin{array}{l}
U=\left\{(u_1(\cdot),u_2(\cdot)) \, | \, 
u_i(\cdot) \mbox{ is Lebesgue measurable and }
0 \leq u_i(t)\leq 1, \  t\in[0, t_f]\right\}.
\end{array}
\end{equation}
The admissible controls are restricted/constrained as $0\leq u_{i}(t)\leq 1$
because they are fractions of the  biological control and the cost of advertisements. 
We assume that the controls are bounded between $0$ and $1$. In the context of this assumption, 
when the controls take the minimum value $0$, it means no extra measures are implemented for 
the reduction of  the pest from the environment; while the maximum value $1$ corresponds 
to 100\% successfully implementation of the protection for the pest eradication.


\subsection{Existence of solution}
\label{sec4.1}

The existence of an optimal control pair 
can be guaranteed by using the results in the book of 
Fleming and Rishel \cite{Fleming}. 

\begin{theorem}
\label{thm:exist:ocs}
The optimal control problem, defined by
the objective functional \eqref{obj} on the admissible set \eqref{control:set}
subject to the control system \eqref{state2}
and initial conditions 
\((X_{0},S_{0},I_{0},A_{0})\geq(0,0,0,0)\), 
admits a solution
$\left(u_{1}^*,u_{2}^*\right)$ (a pair of functions) such that
$$
J(u_{1}^*,u_{2}^*)
=\min_{(u_{1},u_{2})\in U} J(u_1,u_2).
$$
\end{theorem}

\begin{proof}
All the state variables involved in the model are continuously differentiable. 
Therefore, the result follows if the following conditions are met 
(see \cite{Fleming} and pages 98--137 of \cite{Suzan}):
\begin{itemize}
\item[(i)] The set of solutions to the system \eqref{state2} 
with control variables in \eqref{control:set} are non-empty;
\item[(ii)] The control set \({U}\) is convex and closed;
\item[(iii)] Each right hand side of the state system \eqref{state2} is continuous,  
bounded above, and can be written as a linear function of \(u\)
with coefficients depending on time and the state;
\item[(iv)] The integrand \(g(t,Y,u)=A_{1}S^{2}-A_{2}A^{2}+\frac{1}{2}B_{1}u_{1}^{2}+\frac{1}{2}B_{2}u_{2}^{2}\) 
of the objective functional \eqref{obj} is convex;
\item[(v)] There exist positive numbers \(\ell_{1},\ell_{2},\ell_{3}\) 
and a constant \(\ell>1\) such that 
\[
g(t,S,A,u_{1},u_{2})\geq-\ell_{1}
+\ell_{2}|u_{1}|^{\ell}+\ell_{3}|u_{2}|^{\ell}.
\]
\end{itemize}	
We justify each one of these items.
\begin{itemize}
\item[(i)] Since \({U}\) is a nonempty set of real valued measurable 
functions on the finite time interval \(0\leq t_{f}\),
the system \eqref{state2} has bounded coefficients 
and hence any solutions are bounded on \([0,t_{f}]\).
It follows that the corresponding solutions for 
system \eqref{state2} exist.
\item[(ii)] We begin to prove the convexity of the control set 
${U}=\left\{u\in\mathbb{\mathbb{R}}^{2}\colon||u||\leq1\right\}$.
Let \(u_{1},u_{2}\in {U}\) such that \(||u_{1}||\leq1\) and \(||u_{2}||\leq1\). 
Then, for an arbitrary \(\lambda\in[0,1]\), we have
\begin{equation*}
\begin{split}
||\lambda u_{1}+(1-\lambda)\,u_{2}||
&\leq||\lambda\,u_{1}||+||(1-\lambda)\,u_{2}||
=|\lambda|\,||u_{1}||+|(1-\lambda)|\,||u_{2}||
=\lambda\,||u_{1}||+(1-\lambda)\,||u_{2}||\\
&\leq\lambda\,(1)+(1-\lambda)\,(1)
=\lambda\,(1)+1-\lambda\,(1)
= 1.
\end{split}
\end{equation*}
Hence, \(||\lambda\,u_{1}+(1-\lambda)u_{2}||\leq1\). 
This implies that the control set \({U}\) is convex.
The control space 
\({U}=\left\{\left(u_{1},u_{2}\colon u_{1},u_{2}\right)\right\}\) 
is measurable, $0\leq {u_1}_{\min}\leq u_{1}(t)\leq {u_{1}}_{\max}\leq1$, 
$0\leq {u_{2}}_{\min}\leq u_{2}(t)\leq {u_{2}}_{\max}\leq1$
is closed by definition. Therefore, \({U}\) is convex and closed set.
\item[(iii)] All the right-hand sides of equations of system \eqref{state2} are
continuous, all variables \(X,S,I,A\) and \(u\) are bounded on \([0,t_{f}]\),
and can be written as a linear function of $u_{1}$ and $u_{2}$ 
with coefficients depending on the time and state.
\item [(iv)] Let \(\lambda\in[0,1]\) and $a=(a_{1},a_{2})$,
$b=(b_{1},b_{2})\in {U}$. Then, 
\begin{equation*}
\begin{aligned}
g&\left(t,Y,(1-\lambda)\,a+\lambda\,b\right)
-\left((1-\lambda)\,g(t,Y,a)+\lambda\,g(t,Y,b)\right)\\
&=\frac{B_{1}}{2}\left[\left(1-\lambda\right)^{2}\,{a_{1}}^{2}+2\,\lambda\,
\left(1-\lambda\right)\,a_{1}\,b_{1}+\lambda^{2}\,{b_{1}}^{2}\right]\\
&\quad +\frac{B_{2}}{2}\left[\left(1-\lambda\right)^{2}\,{a_{2}}^{2}+2\,\lambda\,
\left(1-\lambda\right)\,a_{2}\,b_{2}+\lambda^{2}\,{b_{2}}^{2}\right]\\
&\quad -\left[\left(1-\lambda\right)\left(\frac{B_{1}}{2}{a_{1}}^{2}
+\frac{B_{2}}{2}{a_{2}}^{2}\right)\right]
-\lambda\left(\frac{B_{1}}{2}{b_{1}}^{2}+\frac{B_{2}}{2}{b_{2}}^{2}\right)\\
&=\left(\lambda^{2}-\lambda\right)\left(\frac{B_{1}}{2}\left(a_{1}-b_{1}\right)^{2}
+\frac{B_{2}}{2}\left(a_{2}-b_{2}\right)^{2}\right)\\
&=\frac{\left(\lambda^{2}-\lambda\right)}{2}\left(a-b\right)^{2}\\
&=\frac{1}{2}\,\lambda\,\underset{{\leq0}}{\underbrace{\left(\lambda-1\right)}}\,
\left(a-b\right)^{2}\leq 0.
\end{aligned}
\end{equation*}
Hence, \(g(t,Y,(1-\lambda)a+\lambda\,b)\leq(1-\lambda)g(t,Y,a)+\lambda\,g(t,Y,b)\).
Therefore, \(g(t,Y,u)\) is convex in the control set \({U}\).
\item [(v)] Finally, it remains to show that there exists a constant 
\(\ell^{*}>1\) and positive constants \(\ell_{1},\ell_{2}\) and \(\ell_{3}\) such that
\[
\frac{B_1\,{u_{1}}^{2}(t)}{2}+\frac{B_2\,{u_{2}}^{2}(t)}{2}+A_{1}\,
S^{2}-A_{2}\,{A^{2}}\geq-\ell_{1}+\ell_{2}|u_{1}|^{\ell}+\ell_{3}|u_{2}|^{\ell}.
\]
In \mref{sec2}, we already showed that the state variables are bounded.
Let $\ell_{1}=\sup\left(A_{1}\,S^{2}-A_{2}\,A^{2}\right)$,
$\ell_{2}=B_{1},\ell_{3}=B_{2}$ and \(\ell=2\). Then it follows that
\[
\frac{B_1\,{u_{1}}^{2}(t)}{2}+\frac{B_2\,{u_{2}}^{2}(t)}{2}
+A_{1}\,S^{2}-A_{2}\,{A^{2}}\geq-\ell_{1}+\ell_{2}|u_{1}|^{\ell}
+\ell_{3}|u_{2}|^{\ell}.
\]
\end{itemize}
We conclude that there exists an optimal control pair.
\end{proof}


\subsection{Characterization of the optimal controls}
\label{sec4.2}

Since there exists an optimal control pair for minimizing functional \eqref{obj}
subject to system \eqref{state2}, we now derive necessary conditions to determine
the optimal control pair by using the Pontryagin minimum principle \cite{Fleming}.

\begin{theorem}
\label{ourThm:fromPMP}
Let  $\left(u_1^*(t),u_2^*(t)\right)$, $t \in [0, t_f]$, be an optimal control pair
that minimizes $J$ over $U$ for which it corresponds the states $X^*,S^*,I^*,A^*$,
solution to system \eqref{state2}. Then, there exist adjoint functions
$\lambda_1$, $\lambda_2$, $\lambda_3$, $\lambda_4$ satisfying the system
\begin{equation}
\label{eq:adjSyst}
\begin{cases}
\displaystyle \frac{d\lambda_1(t)}{dt}
=\lambda_1(t)\left(\frac{\alpha cS^*(t)}{(c+X^*(t))^2}+\frac{\phi\alpha cI^*(t)}{(c+X^*(t))^2}
-r\left(1-\frac{2X^*(t)}{K}\right)\right)-\lambda_2(t)\frac{m_1\alpha cS^*(t)}{(c+X^*(t))^2}
-\lambda_3(t)\frac{m_2\phi\alpha cI^*(t)}{(c+X^*(t))^2},\\[0.3cm]
\displaystyle \frac{d\lambda_2(t)}{dt}=-2A_1S^*(t)+\lambda_1(t)\frac{\alpha X^*(t)}{c+X^*(t)}
+\lambda_2(t)\left(u_1^*(t)\frac{\lambda A^*(t)}{(a+A^*(t))}-\frac{m_1\alpha X^*(t)}{c+X^*(t)}+d\right)\\[0.3cm]
\qquad \qquad -\displaystyle \lambda_3(t)\frac{u_1^*(t)\lambda A^*(t)}{a+A^*(t)}-\lambda_4(t)\sigma,\\[0.3cm]
\displaystyle \frac{d\lambda_3(t)}{dt}=\lambda_1(t)\frac{\phi\alpha X^*(t)}{c+X^*(t)}
+\lambda_3(t)\left(d+\delta-\frac{m_2\phi\alpha X^*(t)}{c+X^*(t)}\right)-\lambda_4(t)\sigma,\\[0.3cm]
\displaystyle \frac{d\lambda_4(t)}{dt}=2A_2A^*(t)+\lambda_2(t)\frac{u_1^*(t)\lambda a S^*(t)}{(a+A^*(t))^2}
-\lambda_3u_1^*(t)\frac{\lambda aS^*(t)}{(a+A^*(t))^2}+\lambda_4(t)\eta,
\end{cases}
\end{equation}
subject to the terminal conditions
\begin{equation}
\label{transv:cond}
\lambda_i(t_f)=0, \quad i=1,2,3,4.
\end{equation}
Furthermore, the optimal control functions $u_1^*$ and $u_2^*$ are characterized by
\begin{equation}
\label{u1}
u_1^*(t)=\max\left\{0,\min\left\{1,\frac{(\lambda_2(t)-\lambda_3(t))
\lambda A^*(t) S^*(t)}{B_1(a+A^*(t))}\right\}\right\}
\end{equation}
and
\begin{equation}
\label{u2}
u_2^*(t)=\max\left\{0,\min\left\{1,-\frac{\lambda_4(t)\gamma}{B_2}\right\}\right\},
\end{equation}
respectively, for $t \in [0, t_f]$.
\end{theorem}

\begin{proof}
Let $(u_1^*,u_2^*)$ be optimal controls whose existence
is assured by \mref{thm:exist:ocs}.
We first define the Hamiltonian function by
\begin{multline*}
H\left(t,u,X,S,I,A,\lambda_1,\lambda_2,\lambda_3,\lambda_4\right)
= A_1S^2-A_2A^2+\frac{1}{2}B_1u_1^2+\frac{1}{2}B_2u_2^2
+\lambda_1\left(r X\left[1 - \frac{X}{K}\right] - \frac{\alpha XS}{c+X}- \frac{\phi \alpha XI}{c+X}\right)\\
+\lambda_2\left(\frac{m_1 \alpha XS}{c+X} - u_1 \frac{\lambda AS}{a+A} -dS\right)
+\lambda_3\left(\frac{m_2 \phi\alpha XI}{c+X} + u_1 \frac{\lambda AS}{a+A}- (d+\delta)I\right)
+\lambda_4\left(u_2\gamma+ \sigma(S+I) - \eta A\right),
\end{multline*}
where $\lambda_1$, $\lambda_2$, $\lambda_3$, $\lambda_4 $ are the adjoint multipliers.
The Pontryagin Minimum Principle (PMP) \cite{Fleming} provides necessary optimality
conditions that $(u_1^*,u_2^*)$ must satisfy. Roughly speaking,
the PMP reduces the optimal control problem (a dynamic optimization problem)
into one of minimizing the Hamiltonian $H$ in the space of the values of the controls
(a static optimization problem). In concrete, taking the derivative of the Hamiltonian
with respect to $X$, $S$, $I$ and $A$, the adjoint sytem
\begin{equation*}
\begin{cases}
\displaystyle \frac{d\lambda_1}{dt}=-\frac{\partial H}{\partial X}
=\lambda_1\left(\frac{\alpha cS}{(c+X)^2}+\frac{\phi\alpha cI}{(c+X)^2}
-r\left(1-\frac{2X}{K}\right)\right)-\lambda_2\frac{m_1\alpha cS}{(c+X)^2}
-\lambda_3\frac{m_2\phi\alpha cI}{(c+X)^2},\\[0.3cm]
\displaystyle \frac{d\lambda_2}{dt}=-\frac{\partial H}{\partial S}
=-2A_1S+\lambda_1\frac{\alpha X}{c+X}
+\lambda_2\left(u_1\frac{\lambda A}{a+A}-\frac{m_1\alpha X}{c+X}+d\right)
-\lambda_3\frac{u_1\lambda A}{a+A}-\lambda_4\sigma,\\[0.3cm]
\displaystyle \frac{d\lambda_3}{dt}=-\frac{\partial H}{\partial I}
=\lambda_1\frac{\phi\alpha X}{c+X}+\lambda_3\left(d
+\delta-\frac{m_2\phi\alpha X}{c+X}\right)-\lambda_4\sigma,\\[0.3cm]
\displaystyle \frac{d\lambda_4}{dt}=-\frac{\partial H}{\partial A}
=2A_2A+\lambda_2\frac{u_1\lambda a S}{(a+A)^2}
-\lambda_3u_1\frac{\lambda aS}{(a+A)^2}+\lambda_4\eta,
\end{cases}
\end{equation*}
of PMP give us \eqref{eq:adjSyst}, while the transversality conditions
$\lambda_1(t_f)=\lambda_2(t_f)=\lambda_3(t_f)=\lambda_4(t_f)=0$
of PMP give \eqref{transv:cond}. Moreover, the minimality condition
of PMP asserts that the optimal controls $u_1^*$ and $u_2^*$ must satisfy
\begin{align}
\label{control1}
u_1^*(t) &=
\begin{cases}
\displaystyle
\frac{(\lambda_2(t)-\lambda_3(t))\lambda A^*(t)S^*(t)}{B_1(a+A^*(t))}
&\text{ if } \ 0 \leq
\displaystyle \frac{(\lambda_2(t)-\lambda_3(t))
\lambda A^*(t)S^*(t)}{B_1(a+A^*(t))}\leq 1,\\[0.3cm]
\displaystyle \ 0 &\text{ if } \
\displaystyle \frac{(\lambda_2(t)-\lambda_3(t))
\lambda A^*(t)S^*(t)}{B_1(a+A^*(t))}\leq 0, \\[12pt]
\displaystyle \ 1 &\text{ if } \
\displaystyle \frac{(\lambda_2(t)-\lambda_3(t))
\lambda A^*(t)S^*(t)}{B_1(a+A^*(t))}\geq 1,
\end{cases}\\
\label{control2}
u_2^*(t) &=
\begin{cases}
\displaystyle -\frac{\lambda_4(t)\gamma}{B_2}
&\text{ if } \
0\leq -\displaystyle \frac{\lambda_4(t)\gamma}{B_2}\leq 1, \\[12pt]
\ 0 &\text{ if } \
\displaystyle -\frac{\lambda_4(t)\gamma}{B_2}\leq 0, \\[12pt]
\ 1 &\text{ if } \
\displaystyle -\frac{\lambda_4(t)\gamma}{B_2}\geq 1.
\end{cases}
\end{align}
In compact notation, \eqref{control1} and \eqref{control2}
are written, equivalently, respectively as in
\eqref{u1} and \eqref{u2}.
\end{proof}
As a result, the optimal control $(u_1^*,u_2^*)$ that minimizes the given purpose
functional over the control set $U$ is given by \eqref{u1} and \eqref{u2}.
Since the solutions of the state system \eqref{state2}
and adjoint system \eqref{eq:adjSyst} are bounded and satisfy Lipschitz conditions,
the optimality system has a unique solution for some small time $t_f$.
Thus, a restriction on the length of the time interval $[0,t_f]$ in the control problem
pledges distinctiveness of the optimality system \cite{dynamic}.


\section{Numerical simulations}
\label{sec5}

In this section, we give some numerical simulations
for the analytical solutions of systems \eqref{state1}
and \eqref{state2}. Our numerical solutions show how realistic
our results are and illustrate well the obtained analytical results.
We begin by analyzing system \eqref{state1}, without controls,
then  our model \eqref{state2} subject to the optimal controls,
as characterized by Pontryagin Minimum Principle (PMP). Our numerical simulations
are attained with a set of parameter values as given in \mref{table1.1}.
\begin{table}[H]
\begin{center}
\caption{Parameter values used in our numerical simulations.}
\setlength{\tabcolsep}{4mm} 
{
\begin{tabular}{clll} \hline
Parameters  & Description  & Value &Source\\ \hline
$r$ & Growth rate of crop biomass & 0.1 per day &\cite{JTB} \\
$K$ & Maximum density of crop biomass & 1 $m^{-2}$ &\cite{IJB}\\
$\lambda$ & Aware people activity rate& 0.025 per day&\cite{JTB}\\
$d$ & Natural mortality of pest& 0.01 day$^{-1}$&\cite{EC2}\\
$m_1$ & Conversion efficacy of susceptible pests & 0.8 &\cite{JTB}\\
$m_2$& Conversion efficiency of infected pest& 0.6& \cite{JTB}\\
$\delta$ & Disease related mortality rate & 0.1 per day &\cite{IJB}\\
$a$ & Half saturation constant &  0.5 & \cite{JK}\\
$\alpha$ & Attack rate of pest & 0.025 pest$^{-1}$per day&\cite{JTB}\\
$\sigma$ & Aware people growth rate& 0.015 per day& Assumed\\
$\eta$ & Fading of memory of aware people& 0.015 {day$^{-1}$}& \cite{JTB} \\
$\gamma$ & Rate of awareness from global source & 0.003 day$^{-1}$ &\cite{IJB}\\
$c$ & Half saturation constant& 1&\cite{SPA}\\
\hline
\end{tabular}}
\label{table1.1}
\end{center}
\end{table}
We examine the impact of optimal control strategies by pertaining
a Runge--Kutta fourth-order scheme on the optimality system. The optimality system
is obtained by taking the state system together with the adjoint system, the optimal controls,
and the transversality conditions. The vibrant behavior of the model, in relation to both
controls, is also deliberated. The optimal policy is achieved by finding a solution 
to the state system \eqref{state2} and costate system \eqref{eq:adjSyst}. 
An iterative design is explored and used to decide the solution 
for the optimality system. The numerical method we utilized
is the forward-backward sweep method, which includes the iterative Runge--Kutta fourth-order
progressive-regressive schemes. The progressive scheme is used in obtaining the solutions
of the state ODEs given in \eqref{state2} with given/fixed initial conditions, while
the regressive scheme is applied in obtaining the solutions of the adjoint system given by
\eqref{eq:adjSyst} with the transversality conditions \eqref{transv:cond}.
Following this, we use a convex grouping of the previous iteration approximate controls
and the ones from the characterization values, to update the controls. This procedure
continues and the iterative values are repeated if the values of the unknowns at the preceding
iteration are not considerably near to the ones at the present iteration. We proceed with 
the necessary iterations till convergence is attained. It is a two point boundary-value problem,
with divided boundary conditions at $t_0 = 0$ and $t = t_f$. It elucidates our selection
of the Runge--Kutta fourth order method. The optimality system 
has unique solution for some small time $t_f $ and we use 
$t_f=100$ days for our system with the application of optimal control. 
The time $t_f=2000$ days is used in the system without controls, 
to see the stability switches of the dynamical system 
without applying the optimal controls.
Besides, on the numerical simulations the values of the weight function are taken as
$A_1 =1015$, $A_2=1010$, $B_1=1.6$, and $B_2=1$, and the initial state variables as
$X(0)=0.2$, $S(0)=0.07$, $I(0)=0.05$, $A(0)=0.5$.
In \mref{fig1}, the time series solution of model system \eqref{state1} is sketched.

\begin{figure}[H]
\centering
\includegraphics[scale=0.86]{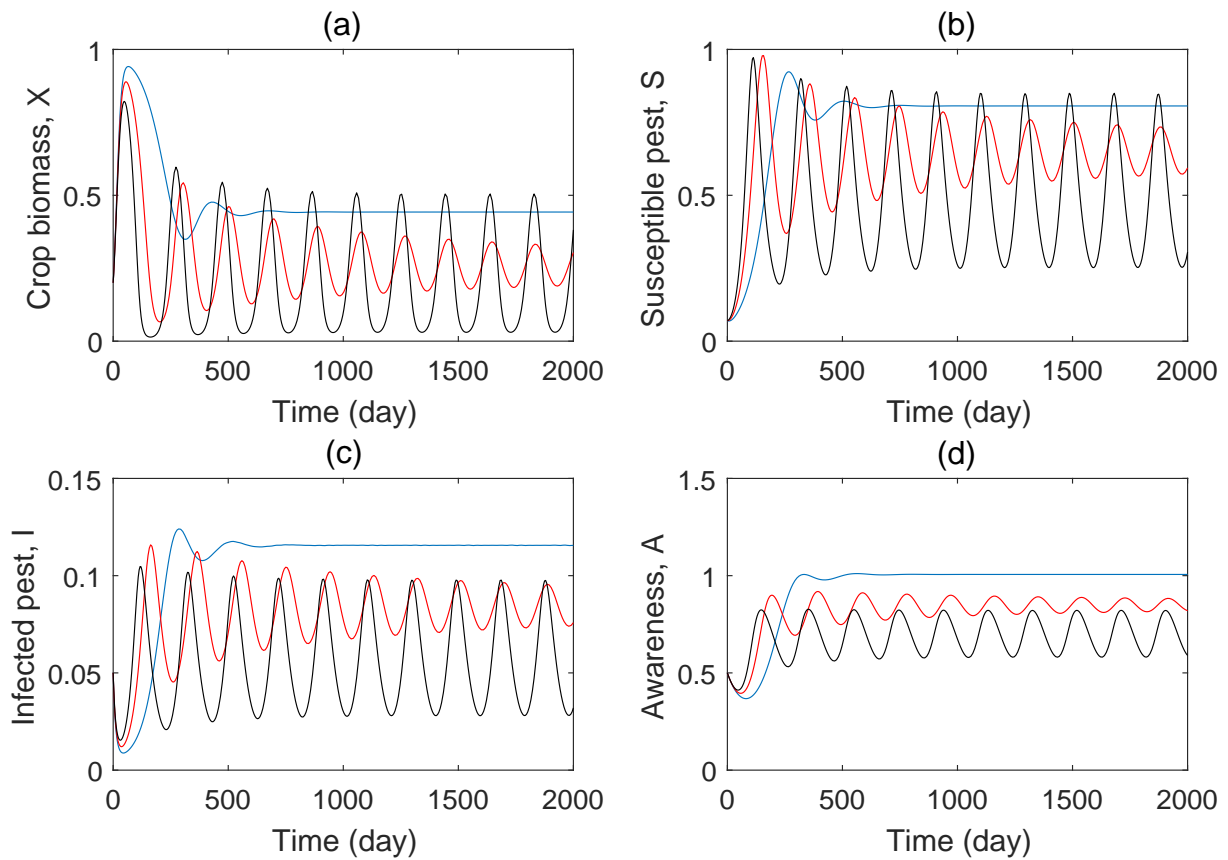}
\caption{Numerical solution of system \eqref{state1} (without controls)
for different values of the rate $\alpha$ of pest:
$\alpha=0.06$ (blue line), $\alpha=0.08$ (red line),
$\alpha=0.06$ (black line).
Other parameter values as in \mref{table1.1}.}
\label{fig1}
\end{figure}

It is observed that our model variables $X(t)$, $S(t)$, $I(t)$ and $A(t)$ become oscillating
as the worth of utilization rate (i.e., $\alpha$) enlarges. Also, steady state value of both
pest population (when exists) are increased as $\alpha$ rises. An oscillating solution is seen
for bigger values of $\alpha$ ($\alpha=0.1$). A bifurcation illustration is shown in 
\mref{fig2}, taking $\alpha$ as the main parameter. 

\begin{figure}[H]
\centering
\includegraphics[scale=0.67]{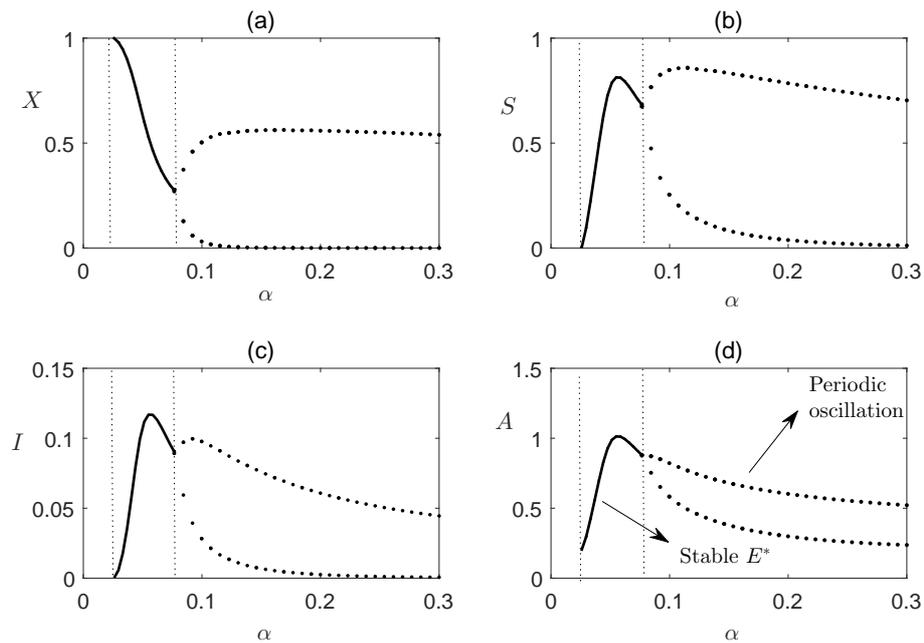}
\caption{Bifurcation diagram of the coexisting equilibrium $E^*$ (when exists) of
system \eqref{state1} (without controls) with respect to the rate $\alpha$ of pest.
Other parameter values as in \mref{table1.1}.}
\label{fig2}
\end{figure}

Decisive values depend on many parameters, such as the employment 
rate of the awareness program (see \mref{fig3}).

\begin{figure}[H]
\centering
\includegraphics[scale=0.67]{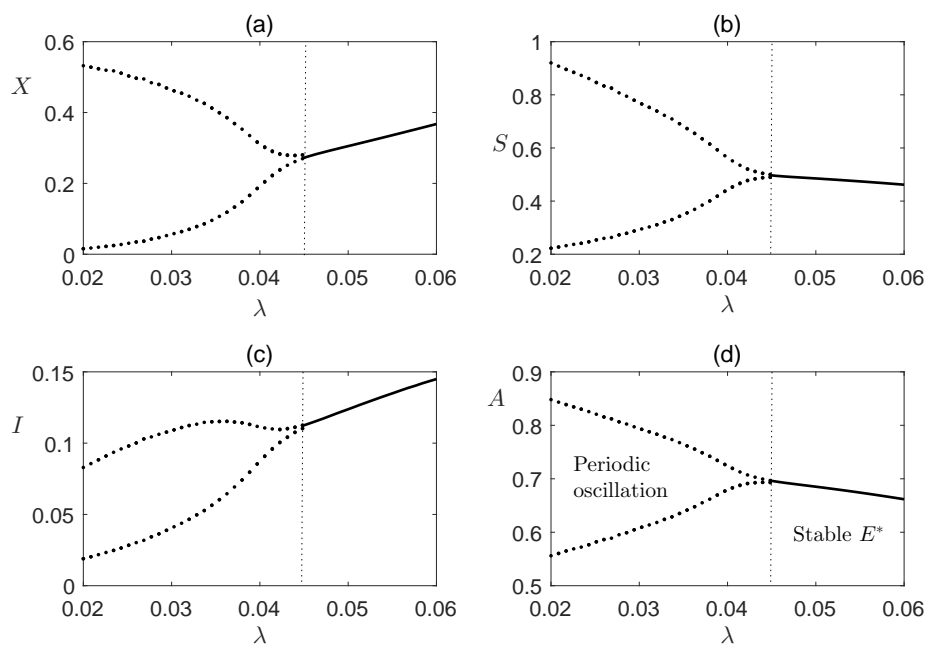}
\caption{Bifurcation diagram of the coexisting equilibrium $E^*$ (when exists)
of system \eqref{state1} with respect to the aware people
activity rate $\lambda$. Other parameters as in \mref{table1.1}.}
\label{fig3}
\end{figure}

The numerical solutions of the control system \eqref{state2}
are given in Figures~\ref{fig4}--\ref{fig6},
showing the impact of optimal control theory.

\begin{figure}[H]
\centering
\includegraphics[scale=0.63]{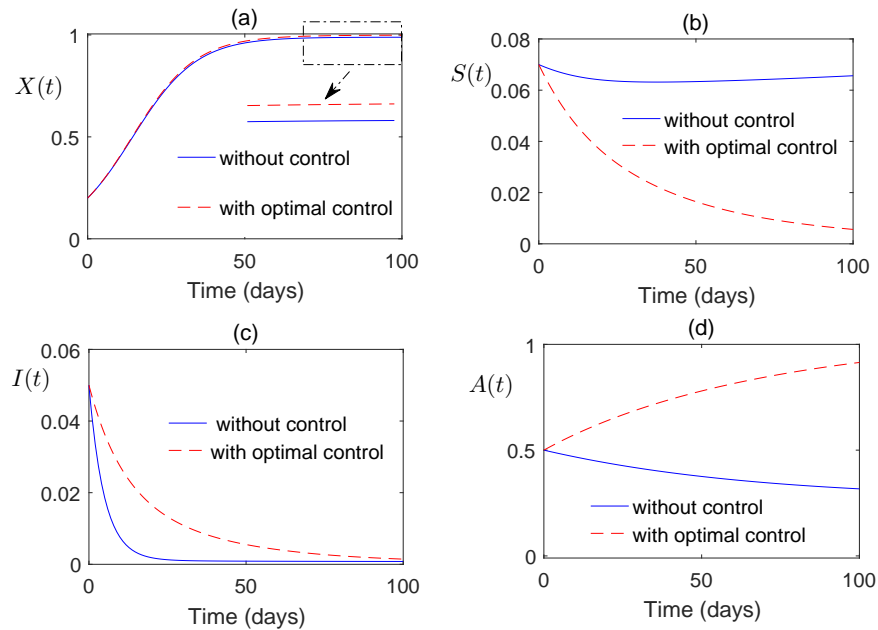}
\caption{Numerical solutions of the controlled system model \eqref{state2}
with $u_1\equiv0$: solution without control (i.e., for $u_2\equiv 0$,
coinciding with the solution of \eqref{state1}), in blue color, versus optimal
solution (i.e., for $u_2\ne 0$ chosen according with the PMP), in red color.}
\label{fig4}
\end{figure}

\begin{figure}[H]
\centering
\includegraphics[scale=0.63]{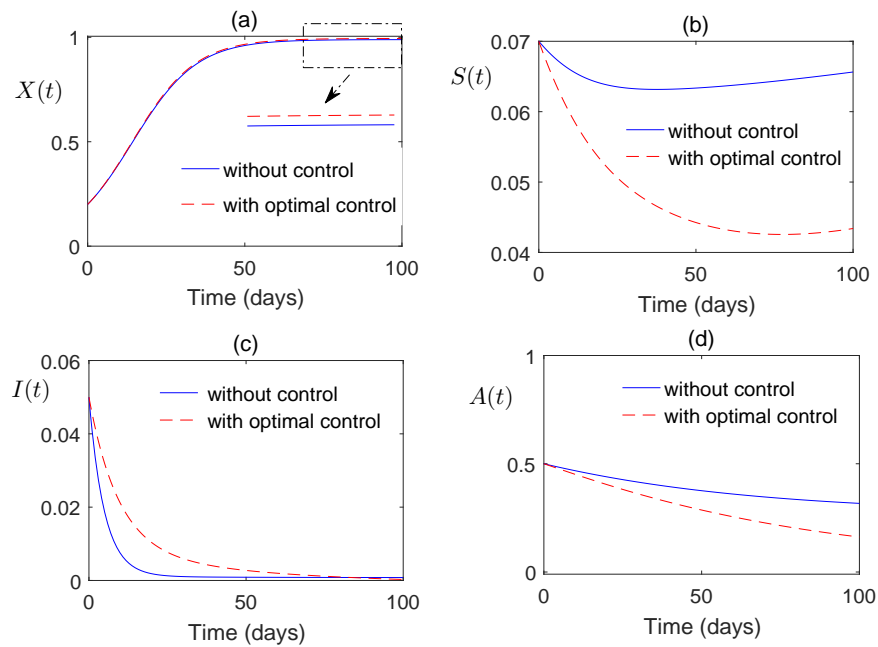}
\caption{Numerical solutions of the controlled system model \eqref{state2}
with $u_2\equiv0$: solution without control (i.e., for $u_1\equiv 0$,
coinciding with the solution of \eqref{state1}), in blue color, versus optimal
solution (i.e., for $u_1\ne 0$ chosen according with the PMP), in red color.}
\label{fig5}
\end{figure}

In \mref{fig6}, we contrast the dynamics with no control and with
optimal control, according with \mref{sec4}.
We operate our controls for the first $100$ days.
Wavering does not happen with the optimal control strategy.
Plant biomass enlarged and the population of pest reduced drastically
with a sway of the finest outlines of universal alertness (i.e., $u_2^*\gamma$)
and consciousness based control movement, $u_1^*\lambda$, as exposed in
\mref{fig6}. 

\begin{figure}[H]
\centering
\includegraphics[scale=0.90]{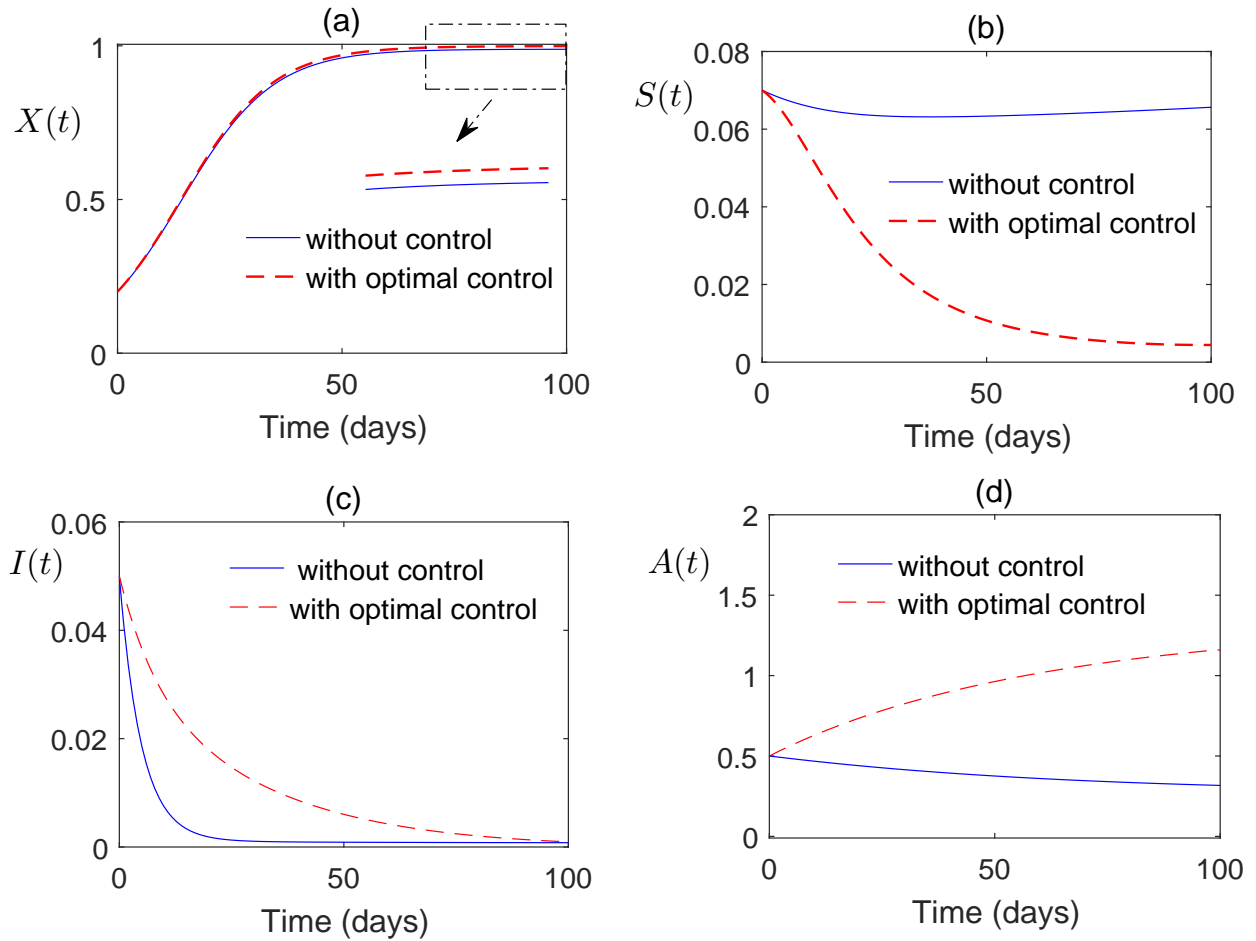}
\caption{Numerical solution of the uncontrolled system \eqref{state1},
in blue color, versus optimal solution (i.e., for both $u_1 = u_1^*$
and $u_2 = u_2^*$ chosen according with \mref{ourThm:fromPMP}), in red color.
The controls $u_1^*$ and $u_2^*$ are shown in \mref{fig7}.}
\label{fig6}
\end{figure}

It is also seen that the susceptible pest population
goes to destruction inside the earliest $80$ days, owing to endeavor
of the extremal controls, which are depicted in \mref{fig7}.
The infected pest population also enhances up to a convinced period and hence declining.
As a result, the outcome of optimal control by means of consciousness based bio-control
has a large input in scheming the pest problem in the plant meadow.

\begin{figure}[H]
\centering
\includegraphics[scale=0.60]{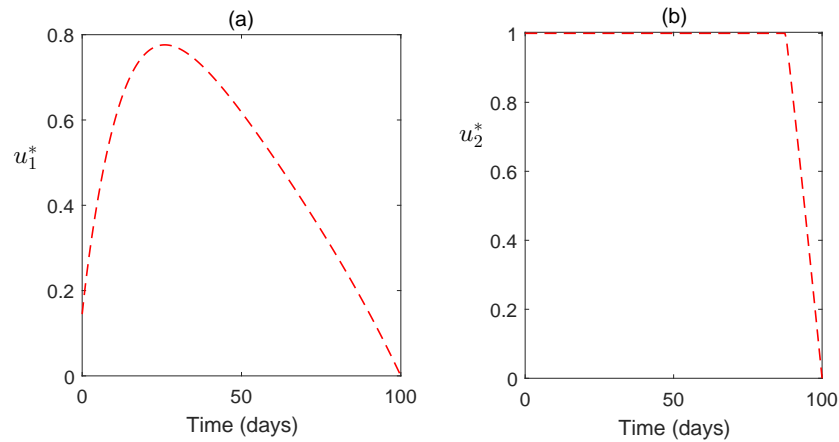}
\caption{Pontryagin extremal controls $u_1^*$ \eqref{u1} and $u_2^*$ \eqref{u2}
plotted as functions of time. The corresponding state trajectories
$X^*$, $S^*$, $I^*$, and $A^*$ are shown, in red color, in \mref{fig6}.}
\label{fig7}
\end{figure}


\section{Discussion and conclusion}
\label{sec6}

In this article, a mathematical model described by a system
of ordinary differential equations has been projected and
analyzed to deal the consequences of attentiveness plans
on the control of pests in farming perform. Our model contains
four concentrations, specifically, concentration of plant-biomass,
density of susceptible pests, infected pest, and population awareness.
We presume that conscious people adopt bio-control measures,
like the included pest administration, since it is environmental and less
harmful to human health. With this loom, the susceptible pest is made infected.
Again, we presume infected pests cannot damage the crop.
We have shown that the proposed model exhibits four steady state points:
(i) the axial fixed point, which is unstable at all time, (ii) the pestles
fixed point, which is stable when the threshold value $R_0$ is less than one
and unstable for $R_0 $ bigger than one, (iii) susceptible pest
free fixed point, which may exist when the carrying capacity $K$ is greater 
than the crop biomass $X$, and (iv) the coexistence fixed point. 
From our analytical and numerical investigations, we saw that the most 
ecological noteworthy parameters in the system are the consumption
rate $\alpha$ and the awareness activity rate $\lambda$. If the collision 
of awareness campaigns increases, the concentration of crop increases as well 
and, as an outcome, pest prevalence declines.

We assumed that responsive groups take on bio-control, such as the
included pest managing, as it is eco-friendly and is fewer injurious
to individual health and surroundings. Neighboring awareness movements
may be full as comparative to the concentration of susceptible pest available
in the plant pasture. We supposed that the international issues, disseminated
by radio, TV, telephone, internet, etc., enlarge the stage of consciousness.
Our study illustrates that if the brunt of alertness movements enlarges,
pest occurrence decreases. Consequently, the density of plant rises.
The system alters its steadiness asymptotically to a cyclic oscillation.
However, global consciousness through radio, TV, and social media
can control the oscillations and makes the system steady.

Moreover, we have used optimal control theory to provide the price effective
outline of bio-pesticides and an universal alertness movement. This loom condensed
the price of administration as well as enlarged the yield. In short, elevated
responsiveness amongst people with optimal outline of bio-pesticides can be a correct
feature for the control of pest in plant field, dropping the solemn matters that
synthetic pesticides have on individual fitness and atmosphere.


\section*{Acknowledgments}

Abraha acknowledges Adama Science and Technology University
for its welcome and bear during this work through the research grant ASTU/SP-R/027/19.
Torres is grateful to the financial support from the Portuguese Foundation
for Science and Technology (FCT), through the R\&D Center CIDMA and project UIDB/04106/2020.
The authors would like to thank three anonymous Reviewers 
for their detailed and thought-out suggestions.


\section*{Conflict of interest}

The authors declare that they have no conflict of interest.



\end{document}